\documentclass[11pt]{amsart}
\usepackage{amsfonts, amssymb, amsmath, mathrsfs}
\usepackage{tikz-cd}

\usepackage[hidelinks]{hyperref}
\usepackage{url}

\usepackage[utf8]{inputenc}

\usepackage{bbm}

\tikzset{
	labl/.style={anchor=south, rotate=90, inner sep=.5mm}
}

\newcommand{\C}{\mathbb{C}}
\newcommand{\F}{\mathbb{F}}

\newcommand{\Fq}{\mathbb{F}_q}

\newcommand{\Fqn}{\mathbb{F}_{q^n}}
\newcommand{\Fqs}{\mathbb{F}_{q^s}}
\newcommand{\Fqsp}{\mathbb{F}_{q^{t}}}

\newcommand{\LL}{\mathbb{L}}
\newcommand{\KK}{\mathbb{K}}
\newcommand{\Q}{\mathbb{Q}}
\newcommand{\Z}{\mathbb{Z}}
\newcommand{\N}{\mathbb{N}}
\newcommand{\R}{\mathbb{R}}
\newcommand{\X}{\mathbb{X}}
\newcommand{\El}{E_\lambda}
\newcommand{\Elx}{\El^{(x_0)}}
\newcommand{\Elxp}{\El^{(x_0')}}

\newcommand{\Ql}{\mathbb{Q}_{\ell}}
\newcommand{\Qlbar}{\overline{\mathbb{Q}}_{\ell}}
\newcommand{\Qpbar}{\overline{\mathbb{Q}}_{p}}

\newcommand{\Qq}{\mathbb{Q}_{q}}

\newcommand{\Qqs}{\mathbb{Q}_{q^s}}
\newcommand{\Qqsp}{\mathbb{Q}_{q^{t}}}
\newcommand{\Qlpbar}{\overline{\mathbb{Q}}_{\ell'}}
\newcommand{\Qbar}{\overline{\mathbb{Q}}}
\newcommand{\Gm}{\mathbb{G}_m}

\newcommand{\calE}{\mathcal{E}}
\newcommand{\calF}{\mathcal{F}}
\newcommand{\calG}{\mathcal{G}}
\newcommand{\calH}{\mathcal{H}}
\newcommand{\calL}{\mathcal{L}}
\newcommand{\calO}{\mathcal{O}}

\newcommand{\calM}{\mathcal{M}}

\newcommand{\calV}{\mathcal{V}}

\newcommand{\Weil}{\mathbf{Weil}}
\newcommand{\VVec}{\mathbf{Vec}}
\newcommand{\bfC}{\mathbf{C}}
\newcommand{\Rep}{\mathbf{Rep}}
\newcommand{\Fisoc}{\mathbf{F\textrm{-}Isoc}}
\newcommand{\Isoc}{\mathbf{Isoc}}
\newcommand{\LS}{\mathbf{LS}}
\newcommand{\Coef}{\mathbf{Coef}}
\newcommand{\Foi}{\Fisoc^\textrm{\dag}}
\newcommand{\oi}{\Isoc^\textrm{\dag}}

\DeclareMathOperator{\sms}{ss}

\DeclareMathOperator{\Spec}{Spec}
\DeclareMathOperator{\Gal}{Gal}
\DeclareMathOperator{\Hom}{Hom}

\DeclareMathOperator{\Ker}{Ker}

\DeclareMathOperator{\End}{End}
\DeclareMathOperator{\Aut}{Aut}

\DeclareMathOperator{\GL}{GL}
\DeclareMathOperator{\id}{id}

\DeclareMathOperator{\DGal}{DGal}

\DeclareMathOperator{\cst}{cst}
\DeclareMathOperator{\et}{\acute{e}t}

\newcommand{\pil}{\pi_1^\lambda}
\newcommand{\pie}{\pi_1^{\et}}
\newcommand{\iso}{\xrightarrow{\sim}}

\begin{document}
	
	\newtheorem{theo}[subsubsection]{Theorem}
	\newtheorem*{theo*}{Theorem}
	\newtheorem{conj}[subsubsection]{Conjecture}
	\newtheorem{prop}[subsubsection]{Proposition}
	\newtheorem{lemm}[subsubsection]{Lemma}
	\newtheorem*{lemm*}{Lemma}
	\newtheorem{coro}[subsubsection]{Corollary}

	\theoremstyle{definition}
	\newtheorem{defi}[subsubsection]{Definition}
	\newtheorem*{defi*}{Definition}
	
	\newtheorem{rema}[subsubsection]{Remark}
	\newtheorem{exam}[subsubsection]{Example}
	\newtheorem{nota}[subsubsection]{Notation}
	\newtheorem{cons}[subsubsection]{Construction}
	\numberwithin{equation}{subsubsection}
	
	\title{The monodromy groups of lisse sheaves and overconvergent $F$-isocrystals}
	\author{Marco D'Addezio}
	\date{\today}
	
\address{Max-Planck-Institut für Mathematik, Vivatsgasse 7, 53111, Bonn, Germany}
\email{ daddezio@mpim-bonn.mpg.de}

	\begin{abstract}
		
		It has been proven by Serre, Larsen--Pink and Chin, that over a smooth curve over a finite field, the monodromy groups of compatible semi-simple pure lisse sheaves have “the same” $\pi_0$ and neutral component. We generalize their results to compatible systems of semi-simple lisse sheaves and overconvergent $F$-isocrystals over arbitrary smooth varieties. For this purpose, we extend the theorem of Serre and Chin on Frobenius tori to overconvergent $F$-isocrystals. To put our results into perspective, we briefly survey recent developments of the theory of lisse sheaves and overconvergent $F$-isocrystals. We use the Tannakian formalism to make explicit the similarities between the two types of coefficient objects. 
		
	\end{abstract}
	
	\maketitle

	\tableofcontents
	
	\section{Introduction}
	\subsection{Background}
		L.~Lafforgue in 2002 proved the Langlands reciprocity conjecture for $\GL_r$ over function fields of positive characteristic, \cite{Laf}. This result gives a correspondence between irreducible lisse $\Qlbar$-sheaves over a smooth connected curve over a finite field and cuspidal automorphic representations. The theorem has a $p$-adic counterpart recently proved by Abe in \cite{Abe}. In Abe's work lisse sheaves are replaced by \textit{overconvergent $F$-isocrystals}, previously introduced by Berthelot. The two results prove Deligne's conjecture \cite[Conjecture 1.2.10]{Weil2} for curves. The lack of a Langlands correspondence for higher dimensional varieties (even at the level of the formulation) forced one to generalize Deligne's conjecture reducing geometrically to the case of curves. One of the difficulties is that one cannot rely on a Lefschetz theorem for the étale fundamental group in positive characteristic (see for example \cite[Lemma 5.4]{Esn}). This means that in general, given a smooth variety $X_0$ over a finite field $\Fq$, it does not exist a smooth curve $C_0\subseteq X_0$ over $\Fq$ with the property that every irreducible lisse sheaf over $X_0$ remains irreducible when restricted to $C_0$. On the other hand, for a given lisse sheaf one can find a suitable smooth curve where the lisse sheaf remains irreducible, \cite{Katz}. The same property holds for overconvergent $F$-isocrystals, \cite{AE}. 
	
	\subsection{Main results}
	We refer to lisse sheaves and overconvergent $F$-isocrystals uniformly as \textit{coefficient objects} (Definition \ref{coef-ob:d}). Let $X_0$ be a smooth connected variety over $\Fq$. Suppose that $\calE_0$ is a coefficient object over $X_0$ such that all the eigenvalues of the Frobenii at closed points are algebraic numbers. Thanks to the known cases of Deligne's conjecture (Theorem~\ref{companions-smooth-theorem}), $\calE_0$ sits in an \textit{$E$-compatible system} $\{\calE_{\lambda, 0}\}_{\lambda\in \Sigma}$, where $E$ is a number field. With this we mean that there exists a set $\Sigma$ of finite places of $E$, containing all the finite places which do not divide $p$, and $\{\calE_{\lambda, 0}\}_{\lambda\in \Sigma}$ is a family of pairwise $E$-compatible $\El$-coefficient objects (Definition \ref{comp:d}), one for each $\lambda\in \Sigma$. 
	
	We use the new results presented above to study the problem of $\lambda$-independence of the \textit{monodromy groups}. Let $\F$ be an algebraic closure of $\Fq$ and $x$ an $\F$-point of $X_0$. For every $\lambda\in \Sigma$, let $G(\calE_{\lambda, 0},x)$ be the \textit{(arithmetic) monodromy group} of $\calE_{\lambda, 0}$ (Definition \ref{monodromy-groups-d}) and $\rho_{\lambda,0}$ the tautological representation of $G(\calE_{\lambda, 0},x)$. We prove the following generalization of \cite[Theorem~1.4]{Chin2}.
	
	\begin{theo}[Theorem~\ref{neutral-component-theorem}]
		\label{intro-neutral-component-theorem}	Suppose that for every $\lambda\in \Sigma$ the coefficient object $\calE_{\lambda, 0}$ is semi-simple. 
		After possibly replacing $E$ with a finite extension, there exists a connected split reductive group $G_0$ over $E$ such that for every $\lambda\in \Sigma$ the extension of scalars $G_0\otimes_E \El$ is isomorphic to the neutral component of $G(\calE_{\lambda, 0},x)$. Moreover, there exists a faithful $E$-linear representation $\rho_0$ of $G_0$ and isomorphisms $\varphi_{\lambda,0}: G_{0}\otimes_E E_{\lambda} \xrightarrow{\sim} G(\calE_{\lambda, 0},x)^{\circ}$ for every $\lambda\in \Sigma$ such that $\rho_0\otimes_E E_{\lambda}$ is isomorphic to $\rho_{\lambda,0}\circ \varphi_{\lambda,0}$.
	\end{theo} 
	Note that in Theorem \ref{intro-neutral-component-theorem} we remove from [\textit{ibid.}, Theorem~1.4] the purity and $p$-plain assumptions (cf. §\ref{rational:d}) and we extend it to overconvergent $F$-isocrystals. Chin proves his result exploiting a reconstruction theorem for connected split reductive groups (Theorem \ref{Chin-theorem}). To apply his theorem, he extends the work of Serre in \cite{Serre} on \textit{Frobenius tori} of étale lisse sheaves, \cite[Lemma 6.4]{Chin2}. We further generalize Chin's result on Frobenius tori.
	
	\begin{theo}[Theorem \ref{max-tori:t}] \label{intro-max-tori-LP:t}
	Let $\calE_0$ be an algebraic coefficient object over $X_0$. There exists a Zariski-dense subset $\Delta\subseteq X(\F)$ such that for every $\F$-point $x\in \Delta$ and every object $\calF_0\in \langle \calE_0 \rangle$, the Frobenius torus $T(\calF_0,x)$ is a maximal torus of $G(\calF_0,x)$. Moreover, if $\calG_0$ is a coefficient object which is compatible with $\calE_0$, the subset $\Delta$ satisfies the same property for every object in $\langle \calG_0 \rangle$.
\end{theo}
	
	To prove the theorem, we extend first Serre--Chin result on the finiteness of conjugacy classes of Frobenius tori (Corollary \ref{finiteness-conj-Frob:c}). To do this we exploit the known cases of Deligne's conjecture. Using Serre's technique, this allows to prove Theorem \ref{intro-max-tori-LP:t} for algebraic étale lisse sheaves. In order to extend Theorem \ref{intro-max-tori-LP:t} to overconvergent $F$-isocrystals and to Weil lisse sheaves that are not étale, we use a \textit{dimension data} argument due to Larsen and Pink (Proposition \ref{max-tori-LP:p}). Thanks to Theorem \ref{intro-max-tori-LP:t}, we are able to prove Theorem \ref{intro-neutral-component-theorem} following Chin's method. Theorem \ref{intro-max-tori-LP:t} is also used in \cite{AD18} as a starting point to prove a certain rigidity result for \textit{convergent $F$-isocrystals} which admit an overconvergent extension. An additional outcome of Theorem \ref{intro-max-tori-LP:t} is provided by the following semi-simplicity statement for the Frobenii at closed points.
	
	\begin{coro}[Corollary \ref{semi-simple-at-points:c}]
		\label{intro-semi-simple-at-points:c}
Let $\calE_0$ be a semi-simple $\Qlbar$-coefficient object. The set of closed points $x_0'$ of $X_0$ where the Frobenius $F_{x'_0}$ is regular semi-simple\footnote{Cf. \cite[§12.2]{Bor91}.} is Zariski-dense in $X_0$.
	\end{coro}
	
In the article, we also generalize the result of Serre and Larsen--Pink on the $\lambda$-independence of the $\pi_0$ of the monodromy groups.
	
	\begin{theo}[Proposition~\ref{connected-components-etale-proposition} and Theorem~\ref{Larsen-and-Pink-theorem}]
		\label{intro-Larsen-Pink-theorem}
	For every $\El$-coefficient object $\calE_0$, the finite group schemes $\pi_0(G(\calE_0,x))$ and $\pi_0(G(\calE,x))$ are constant group schemes. If $\calF_0$ is a coefficient object compatible with $\calE_0$, then there exist canonical isomorphisms $\pi_0(G(\calE_0,x))\iso \pi_0(G(\calF_0,x))$ and $\pi_0(G(\calE,x))\iso \pi_0(G(\calF,x))$ as abstract finite groups.
	\end{theo}
	
	To prove such a theorem for overconvergent $F$-isocyrstals we have to relate their monodromy groups with the étale fundamental group of $X_0$. This is done in §\ref{comparison-etale:ss} and it relies on some previous work done by Crew in \cite{CrewMon}. The rest of the proof follows \cite[Proposition 2.2]{LP5}. Finally, we obtain an independence result for the “Lefschetz theorem” for coefficient objects.
	\begin{theo}
	\label{intro-isoindipendent:t}Let $f_0:(Y_0,y)\to (X_0,x)$ be a morphism of smooth geometrically connected pointed varieties. Let $\calE_0$ and $\calF_0$ be compatible  geometrically semi-simple coefficient objects over $X_0$. Let $\varphi_0: G(f_0^*\calE_0,y)\to G(\calE_0,x)$ and $\psi_0:G(f_0^*\calF_0,y)\to G(\calF_0,x)$ be the morphisms induced by $f_0^*$ and let $\varphi$ and $\psi$ be their restriction to the geometric monodromy groups.
	\begin{itemize}
		\item[{(i)}]If $\varphi$ is an isomorphism, the same is true for $\psi$.
		\item[{(ii)}] If $\varphi_0$ is an isomorphism, the same is true for $\psi_0$.
		
	\end{itemize}
	
\end{theo}	 
	\subsection{Relations with previous works}
	In \cite[Theorem~8.23]{Pal}, Pál gives a proof of a special case of Theorem~\ref{intro-neutral-component-theorem} for curves. It relies on a strong Chebotarev's density theorem for overconvergent $F$-isocrystals [\textit{ibid.}, Theorem~4.13], which is now proven in \cite{HP}. Using the result on Frobenius tori, we do not use Hartl--Pál's theorem. It is also worth mentioning that Drinfeld proved the independence of the entire arithmetic monodromy groups (not only the neutral component) over $\Qlbar$, \cite{Dri2}. He uses a stronger representation-theoretic reconstruction theorem (see Remark~\ref{work-drinfeld-remark}). With Frobenius tori, we prove in addition the existence of a number field $E$ such that the algebraic group $G(\calE_{\lambda, 0},x)$ is split reductive over $\El$ for every $\lambda\in \Sigma$.

	\subsection{The structure of the article}
	We define in §\ref{coef-objects-ss} the categories of coefficient objects and geometric coefficient objects, and we prove some basic results. We also recall some definitions related to the characteristic polynomials of the Frobenii at closed points, and we show that the property of a Weil lisse sheaf of being étale can be checked looking at one closed point (Proposition \ref{p-plain-etale:p}). 
	
	In §\ref{monodromy-groups-ss}, we define the arithmetic and the geometric monodromy groups of coefficient objects, using the Tannakian formalism. We also introduce the Tannakian fundamental groups classifying coefficient objects and geometric coefficient objects. We present a \textit{fundamental exact sequence} relating these groups (Proposition~\ref{fun-exact-seq:p}). The result is essentially entirely proven in the appendix for general \textit{neutral Tannakian categories with Frobenius}. Then in §\ref{comparison-etale:ss} we show that the groups of connected components of these fundamental groups are isomorphic to the arithmetic and the geometric étale fundamental group (Proposition~\ref{connected-components-etale-proposition}). We also prove a complementary result, namely Proposition \ref{connected-after-cover-proposition}.
	
	 In §\ref{rank1:ss} we recall the main result on rank 1 coefficient objects (Theorem \ref{rank-1-finite-order:t}). We introduce in §\ref{t-c:d} the notion of \textit{twist classes} and we prove some structural properties for them. In §\ref{Weights-ss} we recollect some theorems from Weil II that are now known for coefficient objects of both types. For example, we recall the main theorem on weights (Theorem \ref{theory-of-weights-theorem}). Then we present in §\ref{Deligne-conjectures-ss} the state of Deligne's conjectures. In §\ref{compatible-systems-ss} we give the definition of compatible systems of lisse sheaves and overconvergent $F$-isocrystals and we present a strong form of the companions conjecture, due to the work of Chin (Theorem~\ref{companions-smooth-theorem}).
	
	In §\ref{independence-of-monodromy-section}, we investigate the properties of $\lambda$-independence of the monodromy groups varying in a compatible system of coefficient objects. We start by proving in §\ref{the-group-of-ss} the $\lambda$-independence of the groups of connected components (Theorem \ref{intro-Larsen-Pink-theorem}). In §\ref{Frobenius-tori:ss} we extend the theory of Frobenius tori to algebraic coefficient objects and we prove Theorem \ref{intro-max-tori-LP:t} and Corollary \ref{intro-semi-simple-at-points:c}. In §\ref{the-neutral-component-ss} we prove instead Theorem~\ref{intro-neutral-component-theorem}. Finally, in §\ref{Lefschetz-theorem:ss} we prove Theorem \ref{intro-isoindipendent:t}.

	\subsection{Acknowledgements}
	I am grateful to my advisors Hélène Esnault and Vasudevan Srinivas for the helpful discussions. In particular, I thank Hélène Esnault for having spotted a gap in the original proof of Theorem \ref{isoindipendent:t}. I have also greatly benefited from many conversations with Emiliano Ambrosi on the
	contents of the article. I thank Abe Tomoyuki, who has kindly answered my questions on isocrystals and for all the valuable comments on a draft of this text. I thank Anna Cadoret, Chun Yin Hui, and Kiran Kedlaya for feedback. Finally, I thank the referee for his/her detailed comments and useful suggestions.

	\section{Notation and conventions}
	\subsubsection{}
	
	\label{varieties}
	We fix a prime number $p$ and a positive power $q$. Let $\Fq$ be a field with $q$ elements and $\F$ an algebraic closure of $\Fq$. For every positive integer $s$ we denote by $\Fqs$ the subfield of $\F$ with $q^s$ elements. We say that a separated scheme of finite type over a field $k$ is a \textit{variety} over $k$. When we do not specify $k$ it means that we take $\Fq$ as a base field. Note that for us a variety is not necessarily irreducible. A \textit{curve} for us is a one dimensional variety. We will mostly denote by $X_0$ a smooth variety over $\Fq$. When $X_0$ is connected we will sometimes consider $X_0$ as a variety over $k_{X_0}$, the algebraic closure of $\Fq$ in $\Gamma(X_0,\calO_{X_0})$. We denote by $X$ the extension of scalars $X_0\otimes_{\Fq}\F$ over $\F$. In general, we denote with a subscript $_0$ objects and morphisms defined over $\Fq$, and the suppression of the subscript will mean the base change to $\F$.  
	
	We write $|X_0|$ for the set of closed points of $X_0$. If $x_0$ is a closed point of $X_0$, the degree of $x_0$ is defined to be $\deg(x_0):=[\kappa(x_0):\Fq]$. A variety is said ($\F$-)\textit{pointed} if it is endowed with the choice of an $\F$-point. A \textit{morphism of pointed varieties} $(Y_0,y)\to (X_0,x)$ is a morphism of varieties $Y_0\to X_0$ which sends $y$ to $x$. An $\F$-point $x$ of $X_0$ determines a unique closed point of the variety that we denote by $x_0$. Moreover, the $\F$-point $x$ determines an identification $k_{X_0}=\Fqs$, for some $s\in \Z_{>0}$.
	
	\subsubsection{}
	\label{fields}
	The letter $\ell$ will denote a prime number. In general we allow $\ell$ to be equal to $p$. We fix an algebraic closure of $\Q$ and for every $\ell$ an algebraic closure of $\Ql$. For a number field $E$, we write $|E|_\ell$ for the set of finite places of $E$ dividing $\ell$. We also write $|E|_{\neq p}$ for the union $\bigcup_{\ell \neq p} |E|_\ell$ and $|E|$ for $\bigcup_{\ell} |E|_\ell$. If $\lambda\in |E|$, we denote by $E_{\lambda}$ the $\lambda$-adic completion of $E$ in $\Qlbar$. For a characteristic $0$ field $\KK$, we say that an element $a\in \KK$ is an \textit{algebraic number} if it is algebraic over $\Q$. If $a$ is an algebraic number we say that it is \textit{$p$-plain}\footnote{This is an abbreviation for the expression \textit{plain of characteristic $p$} in \cite{Chin2}.} if it is an $\ell$-adic unit for every $\ell\neq p$. 
	
	\subsection{Tannakian categories and group schemes}
		We will extensively make use the theory of Tannakian categories as presented in \cite{DM}.
	\subsubsection{}
	Let $\KK$ be a field. For every Tannakian category $\bfC$ over $\KK$, we say that an object in $\bfC$ is a \textit{trivial object} if it is isomorphic to $\mathbbm{1}^{\oplus n}$ for some $n\in \N$. We say that an object $V\in\bfC$ is \textit{irreducible} if the only subobjects of $V$ are $0$ and $V$ itself. We say that $V\in\bfC$ is \textit{absolutely irreducible} if for ever finite extension $\LL/\KK$, the extension of scalars $V\otimes_\KK \LL$ is irreducible. A \textit{Tannakian subcategory} of $\bfC$, for us, is a strictly full abelian subcategory, closed under $\otimes$, duals, subobjects (and thus quotients). If $V$ is an object of $\bfC$, we denote by $\langle V\rangle$ the smallest Tannakian subcategory of $\bfC$ containing $V$. We write $\VVec_\KK$ for the Tannakian category of finite dimensional $\KK$-vector spaces.	
	
	\subsubsection{}
	If $\omega$ is a fibre functor of $\bfC$, over an extension $\LL$, we say that the affine group scheme $\underline{\Aut}^{\otimes}(\omega)$ over $\LL$ is the \textit{Tannakian group} of $\bfC$ with respect to $\omega$. We say that the Tannakian group of $\langle V\rangle$ with respect to $\omega$ is the \textit{monodromy group} of $V$ (with respect to $\omega$). If the monodromy group of $V$ is finite, we say that $V$ is a \textit{finite object}.
	
	\subsubsection{}
	For every group scheme $G$, we denote by $\pi_0(G)$ the \textit{group of connected components} of $G$ and by $G^{\circ}$ the \textit{neutral component} of $G$. If $G$ is an algebraic group, the \textit{reductive rank} of $G$ is the dimension of any maximal torus of $G$. When $\KK$ is a characteristic $0$ field, we will say that a morphism $\varphi:G\to H$ of affine group schemes over $\KK$ is \textit{surjective} if it is faithfully flat and we will say that $\varphi$ is \textit{injective} if it is a closed immersion.

	\subsection{Weil lisse sheaves}
	We mainly use the notations and conventions for lisse sheaves as in \cite{Weil2}.
	\subsubsection{}
	If $X$ is a scheme and $x$ is a geometric point of $X$, we denote by $\pi_1^{\et}(X,x)$ the étale fundamental group of $X$. For a finite extension $k/\Fq$ with algebraic closure $\overline{k}$, we say that the inverse of the $q^{[k:\Fq]}$-power Frobenius of $\overline{k}$ is the \textit{geometric Frobenius of $k$} (with respect to $\overline{k}$). We denote by $F$ the geometric Frobenius of $\Fq$ with respect to $\F$. For every $n\in \Z_{>0}$ we denote by $W(\F/\F_{q^n})$ the Weil group of $\F_{q^n}$ (it is generated by $F^n$). We also denote by $W(X_0,x)$ the Weil group of $X_0$. For every closed point $x_0'$ of $X_0$ in the same connected component of $x$ we denote by $F_{x_0'}\subseteq W(X_0,x)$ the conjugacy class of the geometric Frobenius at $x_0'$, as in \cite[§1.1.8]{Weil2}.
	\subsubsection{}
	\label{Weil-forg}
	For every $\ell\neq p$ we have a category $\LS(X,\Qlbar)$ of \textit{lisse $\Qlbar$-sheaves} over $X$, that is the $2$-colimit of the categories $\LS(X,\El)$ of \textit{lisse $\El$-sheaves}, where $\El$ varies among the finite extensions of $\Ql$ in $\Qlbar$ (see \cite[§1.1.1]{Weil2} for more details). If $X_0$ is not geometrically connected over $\Fq$ these categories are not Tannakian (the unit object has too many endomorphisms). If $x$ is a geometric point of $X_0$, we denote by $X^{(x)}$ the connected component of $X$ containing $x$. The categories $\LS(X^{(x)},\El)$ and $\LS(X^{(x)},\Qlbar)$ are neutral Tannakian categories. 
	
	For a lisse $\Qlbar$-sheaf $\calV$ over $X$, an \textit{$n$-th Frobenius structure} on $\calV$ is an isomorphism $(F^n)^*\calV\iso\calV$ with $F$ the geometric Frobenius of $X$. The category of \textit{Weil lisse $\El$-sheaves} over $X_0$, denoted by $\Weil(X_0,\El)$, is defined to be the category of lisse $\El$-sheaves equipped with a ($1$-st) Frobenius structure. We will often refer to Weil lisse sheaves simply as \textit{lisse sheaves}. If $X_0$ is connected, the category $\Weil(X_0,\El)$ is a neutral Tannakian category. 
	
	For every geometric point $x$ of $X_0$ and every $\El$ we have a functor $$\Psi_{x,\El}:\Weil(X_0,\El)\to\LS(X,\El)\to \LS(X^{(x)},\El)$$
	obtained by firstly forgetting the Frobenius structure and then taking the inverse image with respect to the open immersion $X^{(x)}\hookrightarrow X$. If $\calV_0$ is a Weil lisse sheaf, we remove the subscript $_0$ to indicate the lisse sheaf $\Psi_{x,\El}(\calV_0)$.
	
	\subsubsection{}
	\label{Weil-rep}
	The choice of a geometric point $x$ of $X_0$, induces an equivalence between the category of Weil lisse $\Qlbar$-sheaves over $X_0$ and the finite-dimensional continuous $\Qlbar$-representations of the Weil group $W(X_0,x)$, \cite[1.1.12]{Weil2}. We say that a Weil lisse sheaf is an \textit{étale lisse sheaf} if the associated representation of the Weil group factors through the \'etale fundamental group.

	\subsubsection{}
	Notation as in §\ref{Weil-rep}. If $\calV_0$ is a Weil lisse $\El$-sheaf, for every closed point $x_0'\in |X_0|$ we write $P_{x_0'}(\calV_0,t)$ for $\det(1-tF_{x_0'}|\calV_{x})\in \El[t]$ (cf. \cite[1.1.8]{Weil2}). We say that the polynomial $P_{x_0'}(\calV_0,t)$ is the \textit{(Frobenius) characteristic polynomial} of $\calV_0$ at $x_0'$. For every natural number $n$, a lisse $\Qlbar$-sheaf is said to be \textit{pure} of \textit{weight} $n$, if for every closed point $x_0'$ of $X_0$, the eigenvalues of the elements in $F_{x_0'}$ are algebraic numbers and all the conjugates have complex absolute value $(\#\kappa(x_0'))^{n/2}$. If $\iota:\Qlbar\iso \C$ and $w$ is a real number, we say that a lisse sheaf is \textit{$\iota$-pure} of \textit{$\iota$-weight }$w$ if for every closed point $x_0'$ of $X_0$ the eigenvalues of $F_{x_0'}$, after applying $\iota$, have complex absolute value $(\#\kappa(x_0))^{w/2}$. Moreover, we say that a lisse $\Qlbar$-sheaf is \textit{mixed} (resp. \textit{$\iota$-mixed}) if it admits a filtration of lisse $\Qlbar$-sheaf with pure (resp. $\iota$-pure) successive quotients.
		\newpage
	\subsection{Overconvergent \texorpdfstring{$F$-}\ isocrystals }
	\label{notation-isocrystals}
	\subsubsection{}
	\label{p-adic-admissible}
	Let $k$ be a perfect field. We denote by $W(k)$ the ring of $p$-typical Witt vectors over $k$ and by $K(k)$ its fraction field. For every $s\in \Z_{>0}$, we denote by $\Z_{q^s}$ the ring of Witt vectors over $\Fqs$ and by $\Qqs$ its fraction field. We suppose chosen compatible inclusions $\Qqs \hookrightarrow \Qpbar$.

	Let $X_0$ be a smooth variety over $k$, we denote by $\oi(X_0/K(k))$ the category of Berthelot's \textit{overconvergent isocrystals} of $X_0$ over $K(k)$. See \cite{Ber} for a precise definition and \cite{Crew} or \cite{KedIsoc} for a shorter presentation. The category $\oi(X_0/K(k))$ is a $K(k)$-linear rigid abelian $\otimes$-category with unit object $\calO_{X_0}^\dag$, denoted by $K(k)_{X_0}$. The endomorphism ring of $K(k)_{X_0}$ is isomorphic to $K(k)^s$, where $s$ is the number of connected components of $X_0\otimes_k \overline{k}$.

	\subsubsection{}
	We will recall now the notation for the extension of scalars and the Frobenius structure of overconvergent isocrystals. We mainly refer to \cite[§1.4]{Abe}. For every finite extension $K(k) \hookrightarrow \KK$ we denote by $\oi(X_0/K(k))_\KK$ the category of \textit{$\KK$-linear overconvergent isocrystals} of $X_0$ over $K(k)$, namely the category of pairs $(\calM,\gamma)$, where $\calM\in \oi(X_0/K(k))$ and $\gamma: \KK\to \End(\calM)$ is a \textit{$\KK$-structure} (cf. \cite[§1.4.1]{Abe}). We will often omit $\gamma$ in the notation. If $\KK\subseteq \LL$ are finite extensions of $K(k)$ there exists a functor of extension of scalars $$(-)\otimes_\KK \LL:\ \oi(X_0/K(k))_\KK\to \oi(X_0/K(k))_\LL.$$ The category $\oi(X_0/K(k))_\KK$ is a $\KK$-linear rigid abelian $\otimes$-category. When $X_0$ is geometrically connected over $k$, the endomorphism ring of the unit object $\KK_{X_0}:=K(k)_{X_0}\otimes \KK$ is isomorphic to $\KK$.
	
	\subsubsection{}
	\label{inverse-image}
	
	Let us fix our notation on the inverse image functor for overconvergent isocrystals with $\KK$-structure. Let $X_0/\Fqs$ and $Y_0/\Fqsp$ be smooth varieties with $t\geq s$ and let us write $f_0:Y_0/\Fqsp\to X_0/\Fqs$ for a morphism of schemes $f_0:Y_0\to X_0$ which makes the following diagram commutative
	\begin{center}
		\begin{tikzcd}
			Y_0 \arrow{r}{f_0} \arrow[d] & X_0\arrow[d]\\
			\Spec(\Fqsp)  \arrow[r] & \Spec(\Fqs).\
		\end{tikzcd}
	\end{center}
We have an inverse image functor $\widetilde{f_0^*}:\oi(X_0/\Qqs) \to \oi(Y_0/\Qqsp)$ defined in \cite[2.3.2.(iv)]{Ber}. If $\calM$ is an object in $\oi(X_0/\Qqs)$, the overconvergent isocrystal $\widetilde{f_0^*}\calM$ has a canonical $\Qqsp$-structure as an object in $\oi(Y_0/\Qqsp)$. For every finite extension $\Qqsp\subseteq \KK$, the previous functor extends to a $\otimes$-functor $f_0^*:\oi(X_0/\Qqs)_\KK\to \oi(Y_0/\Qqsp)_\KK$ in a natural way. Let us briefly describe it. If $(\calM,\gamma)$ is an object in $\oi(X_0/\Qqs)_\KK$, then $\widetilde{f_0^*}\calM$ is endowed with the $\Qqsp\otimes_{\Qqs}\Qqsp$-structure obtained by making the first copy of $\Qqsp$ acting on $\widetilde{f_0^*}\calM$ via the canonical $\Qqsp$-structure mentioned above and the second copy acting via the restriction of $\widetilde{f_0^*}\gamma$ to $\Qqsp$. The inverse image $f_0^*(\calM,\gamma)$ is then defined to be $(\widetilde{f_0^*}\calM \otimes_{(\Qqsp\otimes_{\Qqs}\Qqsp)} \Qqsp,\widetilde{f_0^*}\gamma\otimes \id_{\Qqsp}).$ 
	
	\subsubsection{}
	
	We write $F: X_0\to X_0$ for the $q$-power Frobenius\footnote{Note that the letter $F$ will denote two different types of Frobenius endomorphisms, depending if we are working with lisse sheaves or isocrystals.}. Let $\KK$ be a finite field extension of $\Qq$. For every $\calM\in\oi(X_0/\Qq)_\KK$ and $n\in \Z_{>0}$, an \textit{$n$-th Frobenius structure} of $\calM$ is an isomorphism between $(F^n)^*\calM$ and $\calM$. We denote by $\Foi(X_0/\Qq)_\KK$ the category of\textit{ overconvergent $F$-isocrystals with $\KK$-structure}, namely the category of pairs $(\calM, \Phi)$ where $\calM\in \oi(X_0/\Qq)_\KK$ and $\Phi$ is a \textit{($1$-st) Frobenius structure} of $\calM$. For every positive integer $n$, the isomorphism $$\Phi_n:=\Phi\circ F^* \Phi \circ \cdots \circ (F^{n-1})^*\Phi$$ will be the \textit{$n$-th Frobenius structure} of $(\calM, \Phi)$. The category $\Foi(X_0/\Qq)_\KK$ is a $\KK$-linear rigid abelian $\otimes$-category. In this case, even if $X_0$ is connected but not geometrically connected over $\Fq$, the ring of endomorphisms of the unit object is isomorphic to $\KK$.
	
	When $X_0$ is a smooth variety over $\Fqs$, for every finite extension $\Qqs\subseteq \KK$, the category of $\KK$-linear isocrystals over $\Qqs$ with $s$-th Frobenius structure is equivalent to the category $\Foi(X_0/\Qq)_\KK$ (see \cite[Corollary 1.4.11]{Abe}). We will use this equivalence without further comments.
	
	\subsubsection{}
	\label{psi-iso}
	The functors of extension of scalars and of inverse image for overconvergent isocrystals with $\KK$-structure extend in the obvious way to overconvergent $F$-isocrystals. If $(X_0,x)$ is a smooth pointed variety over $\Fq$ geometrically connected over $\Fqs$ and $\KK$ is a finite extension of $\Qqs$, the morphism $f_0:X_0/\Fqs\to X_0/\Fq$ which is the identity on $X_0$ induces a functor $$\Psi_{x,\KK}:\Foi(X_0/\Qq)_\KK\to \oi(X_0/\Qqs)_{\KK}$$ that sends $(\calM,\Psi)$ to $f_0^*\calM$. We denote the objects in $\Foi(X_0/\Qq)_{\KK}$ with a subscript $_0$ and we will remove it when we consider the image by $\Psi_{x,{\KK}}$ in $\oi(X_0/\Qqs)_{\KK}$.

	\subsubsection{}
	\label{point:ss}
	For every finite extension $K(k)\subseteq \KK$, the category $\oi(\Spec(k)/K(k))_\KK$ is equivalent to $\VVec_{\KK}$ as a rigid abelian $\otimes$-category. Moreover, if $k\subseteq k'$ is an extension of finite fields, and $K(k')\subseteq \KK$, the Tannakian category $\Foi(\Spec(k')/K(k))_\KK$ is equivalent to the category of (finite-dimensional) $\KK$-vector spaces endowed with an automorphism, which is induced by the Frobenius structure.
	
	\subsubsection{}
	\label{(x0)}
	Let $(X_0,x)$ be a smooth pointed variety over $\Fq$, geometrically connected over $\Fqs$. Let $E_\lambda$ be a finite extension of $\Qqs$ in $\Qpbar$. The category $\oi(X_0/\Qqs)_{\El}$ admits a fibre functor over some finite extension of $\El$. Let $i_0: x_0/\Fqn \hookrightarrow X_0/\Fqs$ be the immersion of the closed point $x_0$ in $X_0$ where $n$ is the degree of $x_0$ (notation as in §\ref{inverse-image}). Let $\Elx$ be the compositum of $\El$ and $\Q_{q^{n}}$ in $\Qpbar$. Then the functor
	
	$$\omega_{x,\El}: \oi(X_0/\Qqs)_{\El} \xrightarrow{\otimes_{\El}\Elx} \oi(X_0/\Qqs)_{\Elx} \xrightarrow{i_0^*} \oi(x_0/\Q_{q^{n}})_{\Elx} \simeq \VVec_{\Elx}$$ is a fibre functor, as proven in \cite[Lemma 1.8]{CrewMon}. This shows that for every finite extension $\El$ of $\Qqs$, the category $\oi(X_0/\Qqs)_{\El}$ is \textit{Tannakian}. Moreover, the composition $\omega_{x,\El}\circ \Psi_{x,\El}$ is a fibre functor for $\Foi(X_0/\Qq)_{\El}$ over $\Elx$, that we will denote by the same symbol. Therefore, for every finite extension $\Qqs\subseteq \El$, the category $\Foi(X_0/\Qq)_{\El}$ is \textit{Tannakian}.\footnote{The category $\Foi(X_0/\Qq)_{\El}$ is actually Tannakian even when $\El$ is simply a finite extension of $\Qq$. For simplicity, in what follows, we will mainly work with finite extensions of $\Qqs$ in order to make $\oi(X_0/\Qqs)_{\El}$ a Tannakian category.}

	\subsubsection{}
	
	Let $i_0:x'_0\hookrightarrow X_0$ be the immersion of a closed point of degree $n$. Let $\KK$ be a finite extension of $\Qq$ and $\LL$ a finite extension of $\KK$ which contains $\Q_{q^{n}}$. For every $\calM_0\in \Foi(X_0/\Qq)_\KK$, we denote by $F_{x_0'}$ the $n$-th Frobenius structure of $i_0^*(\calM_0)\otimes_\KK \LL$. This is the \textit{(linearized geometric) Frobenius} of $\calM_0$ at $x_0'$. By §\ref{point:ss}, it corresponds to a linear automorphism of an $\LL$-vector space. The characteristic polynomial $$P_{x_0'}(\calM_0,t):=\det(1-tF_{x_0'}|i_0^*(\calM_0)\otimes_\KK \LL)\in \KK[t]$$ is the \textit{(Frobenius) characteristic polynomial} of $\calM_0$ at $x_0'$. See \cite[§4.2.1 and §A.3.1]{Abe} for more details.
	
	In analogy with lisse sheaves, we say that overconvergent $F$-isocrystals are pure or $\iota$-pure if they satisfy the similar conditions on the eigenvalues of the Frobenii at closed points and we say that they are mixed or $\iota$-mixed if they have analogous filtrations.
	
	\section{Coefficient objects}
	\label{generalities-section}
	\subsection{First definitions}
	\label{coef-objects-ss}
Following \cite{Ked}, we adopt a notation which allows us to work with lisse sheaves and overconvergent $F$-isocrystals at the same time. Let $X_0$ be a smooth variety over $\F_q$. 
	
	\begin{defi}[Coefficient objects]\label{coef-ob:d}
		For every prime $\ell\neq p$ and every finite field extension $\KK/\Ql$, we say that a Weil lisse $\KK$-sheaf is a \textit{$\KK$-coefficient object}. When $\KK$ is instead a finite field extension of $\Qq$, we say that an object in $\Foi(X_0/\Qq)_{\KK}$ is a \textit{$\KK$-coefficient object}. If $\KK$ is a field of one of the two types, we denote by $\Coef(X_0,\KK)$ the category of $\KK$-coefficient objects. We say that $\KK$ is \textit{the field of scalars} of $\Coef(X_0,\KK)$. For every prime $\ell$, the category of \textit{$\Qlbar$-coefficient objects} is the $2$-colimit of the categories $\Coef(X_0,\El)$ with $\El\subseteq \Qlbar$. It is denoted by $\Coef(X_0,\Qlbar)$.
	\end{defi}

	We will also work with categories of \textit{geometric coefficient objects}. These are built from the categories of coefficient objects by forgetting the Frobenius structure. To get Tannakian categories, in this case, we will put an additional assumption on the field of scalars. Let $(X_0,x)$ be a smooth connected pointed variety over $\Fq$, geometrically connected over $\Fqs$ for some $s\in \Z_{>0}$.

	\begin{defi}[Admissible fields]
		We say that a finite extension of $\Qqs$ is a \textit{$p$-adic admissible field (for $X_0$)}. To uniformize the notation, when $\ell$ is a prime different from $p$, we say that any finite field extension of $\Ql$ is an \textit{$\ell$-adic admissible field}. We will refer to this second kind of fields as \textit{étale admissible fields}. If $E_\lambda$ is an admissible field, we also say that the place $\lambda$ is \textit{admissible}.  
	\end{defi}

	\begin{defi}[Geometric coefficient objects]For every $p$-adic admissible field $\KK$, we defined in §\ref{psi-iso} a functor of Tannakian categories $\Psi_{x,\KK}:\Foi(X_0/\Qq)_{\KK}\to \oi(X_0/\Qqs)_{\KK}$ which forgets the Frobenius structure. We denote by $\Coef(X^{(x)},\KK)$ the smallest Tannakian subcategory of $\oi(X_0/\Qqs)_{\KK}$ containing the essential image of $\Psi_{x,\KK}$. We say that the category $\Coef(X^{(x)}, \KK)$ is the category of \textit{geometric $\KK$-coefficient objects (with respect to $x$)}.

		When $\KK$ is an étale admissible field, we have again a functor $\Psi_{x,\KK}: \Weil(X_0,\KK)\to \LS(X^{(x)},\KK)$ which forgets the Frobenius structure (see §\ref{Weil-forg}). The category of \textit{geometric $\KK$-coefficient objects (with respect to $x$)} is the smallest Tannakian subcategory of $\LS(X^{(x)},\KK)$ containing the essential image of $\Psi_{x,\KK}$ and it is denoted by $\Coef(X^{(x)}, \KK)$.
		
		 For every prime $\ell$, the category of geometric $\Qlbar$-coefficient objects is the $2$-colimit of the categories of geometric $\El$-coefficient objects where $\El$ varies among the admissible fields for $X_0$ in $\Qlbar$. It is denoted by $\Coef(X^{(x)}, \Qlbar)$ and we denote by $\Psi_{x,\Qlbar}$ the functor induced by the functors $\Psi_{x,\KK}$. If $\calE_0$ is a $\Qlbar$-coefficient object, we drop the subscript $_0$ to indicate $\Psi_{x,\Qlbar}(\calE_0)$, thus we write $\calE$ for $\Psi_{x,\Qlbar}(\calE_0)$. When $X_0$ is geometrically connected over $\Fq$ we drop the superscript $^{(x)}$ in the notation for the categories of geometric coefficient objects, as they do not depend on $x$. 
	\end{defi}
	
	\begin{defi}[Geometric properties]Let $\calE_0$ a $\Qlbar$-coefficient object over a smooth connected variety $X_0$ over $\Fq$. We say that $\calE_0$ is \textit{geometrically semi-simple}, \textit{geometrically trivial} or \textit{geometrically finite} if for one (or equivalently any) choice of an $\F$-point $x$, the associated geometric coefficient object $\calE$ is semi-simple, trivial or finite in $\Coef(X^{(x)},\Qlbar)$. When $X_0$ is not connected we say that a coefficient object has one of the previous properties if the restriction to each connected component of $X_0$ does.
	\end{defi}
	
	\begin{defi}[Cohomology of coefficient objects]Let $(X_0,x)$ be a smooth connected pointed variety over $\Fq$, geometrically connected over $\Fqs$ for some $s\in \Z_{>0}$ and let $\calE_0$ be an $E_\lambda$-coefficient over $X_0$. If $\calE_0$ is a lisse sheaf, we denote by $H^i(X^{(x)},\calE)$ (resp. $H^i_c(X^{(x)},\calE)$) the $\lambda$-adic étale cohomology (resp. the $\lambda$-adic étale cohomology with compact support) of $X^{(x)}$ with coefficients in $\calE$ and by $H^0(X_0,\calE_0)$ the fixed points of $(F^s)^*$ acting on $H^0(X^{(x)},\calE)$. When $E_\lambda$ is $p$-adic, we denote by $H^i(X^{(x)},\calE)$ (resp. $H^i_c(X^{(x)},\calE)$) the rigid cohomology (resp. the rigid cohomology with compact support) of $X_0$ with coefficients in $\calE$. We denote by $H^0(X_0,\calE_0)$ the $E_\lambda$-linear subspace of $H^0(X^{(x)},\calE)$ of fixed points under the action of $(F^s)^*$.
		\begin{rema}
			For both types of coefficient objects, if $E_{\lambda,X^{(x)}}$ is the unit object of \break$\Coef(X^{(x)},E_\lambda)$, the $E_\lambda$-vector space $\Hom(E_{\lambda,X^{(x)}},\calE)$ is canonically isomorphic to $H^0(X^{(x)},\calE)$. We also have a canonical isomorphism between $\Hom(E_{\lambda,X_0},\calE_0)$ and $H^0(X_0,\calE_0)$, where $E_{\lambda,X_0}$ is the unit object in $\Coef(X_0,E_\lambda)$.
		\end{rema}
		
	\end{defi}
	\begin{prop}\label{F-equiv:p}
The functor $(F^s)^*$ is a $\otimes$-autoequivalence of $\Coef(X^{(x)}, \El)$. In particular, the category $\Coef(X^{(x)}, \El^{(x_0)})$ endowed with the endofunctor $(F^s)^*$ is a neutral Tannakian category with Frobenius, as defined in \ref{a-defi}.
	\end{prop}
	\begin{proof}
		For lisse sheaves the result is well-known. In the $p$-adic case see \cite[Remark in §1.1.3]{Abe} or \cite[Corollary 6.2]{Laz} for a proof which does not use arithmetic $\mathscr{D}$-modules.
	\end{proof}
	\begin{coro}\label{irreducible-have-F-structure:c}
		Any irreducible object in $\Coef(X^{(x)},\El)$ admits an $n$-th Frobenius structure for some $n\in \Z_{>0}$.
	\end{coro}
	\begin{proof}
		By definition, an irreducible object $\calF$ in $\Coef(X^{(x)},\El)$ is a subquotient of some geometric coefficient object $\calE$ that admits a Frobenius structure. By Proposition \ref{F-equiv:p}, the functor $(F^s)^*$ is an autoequivalence, thus it permutes the isomorphism classes of the irreducible subquotients of $\calE$. This implies that there exists $n>0$ such that $(F^{ns})^* \calF \simeq \calF$, as we wanted.
	\end{proof}
	
	\begin{rema}
		\label{geom-remark}
		
		When $X_0$ is geometrically connected over $\Fq$, the category $\Coef(X, \Qq)$ is the same category as the one considered by Crew to define the fundamental group at the end of §2.5 in \cite{CrewMon}. The author is not aware whether this category is equivalent to the one considered by Abe to define, for example, the fundamental group in \cite[§2.4.17]{Abe}. By Corollary \ref{irreducible-have-F-structure:c}, the category $\Coef(X,\Qq)$ is a Tannakian subcategory of the one defined by Abe.
		
	\end{rema}

	\begin{defi}We say that a $\KK$-coefficient object is \textit{constant} if it is geometrically trivial. We denote by $\Coef_{\cst}(X_0,\El)$ the (strictly) full subcategory of $\Coef(X_0,\El)$ of constant $\El$-coefficient objects. We define similarly $\Coef_{\cst}(X_0,\Qlbar)\subseteq \Coef(X_0,\Qlbar)$.
	\end{defi}

	\begin{defi}For every prime $\ell$, the category $\Coef(\Spec(\F_q),\Qlbar)$ is canonically equivalent to the category of $\Qlbar$-vector spaces endowed with an automorphism. For $a\in \Qlbar^{\times}$ we write $\Qlbar^{(a)}$ for the rank $1$ coefficient object over $\Spec(\F_q)$ associated to the vector space $\Qlbar$ endowed with the multiplication by $a$.
		Let $p_{X_0}:X_0\to \Spec(\Fq)$ be the structural morphism. For every $\Qlbar$-coefficient object $\calE_0$ and every $a\in \Qlbar^\times$, we say that $\calE_0\otimes p_{X_0}^*\left(\Qlbar^{(a)}\right)$ is the \textit{twist} of $\calE_0$ by $a$ and we denote it by $\calE_0^{(a)}$. A twist is said to be \textit{algebraic} if $a$ is algebraic.
	\end{defi}
	
	\begin{rema}
		The operation of twisting coefficient objects by an element $a\in \Qlbar^{\times}$ gives an exact autoequivalence of the category $\Coef(X_0, \Qlbar)$. In particular, for every coefficient object, the property of being absolutely irreducible is preserved by any twist.
	\end{rema}

	\begin{defi}
		If $\calE_0$ is a $\Qlbar$-coefficient object, for every closed point $x_0$ of $X_0$ we denote by $P_{x_0}(\calE_0,t)$ the \textit{(Frobenius) characteristic polynomial} of $\calE_0$ at $x_0$.
	\end{defi}

	\begin{defi}
		\label{rational:d}
		Let $\ell$ be a prime number, $\KK$ a field endowed with an inclusion $\tau: \KK  \hookrightarrow \Qlbar$. We say that a $\Qlbar$-coefficient object $\calE_0$ is \textit{$\KK$-rational with respect to $\tau$} if the characteristic polynomials at closed points have coefficients in the image of $\tau$. 
		A \textit{$\KK$-rational coefficient object} is the datum of $\tau: \KK \hookrightarrow \Qlbar$ and a $\Qlbar$-coefficient object that is $\KK$-rational with respect to $\tau$. We will also say that an $\El$-coefficient object is $E$-rational if it is $E$-rational with respect to the natural embedding $E\hookrightarrow \El\subseteq \Qlbar$. We say that a coefficient object is \textit{algebraic} if it is $\Qbar$-rational for one (or equivalently any) map $\tau:\Qbar\hookrightarrow \Qlbar$. A coefficient object is said \textit{$p$-plain} if it is algebraic and all the eigenvalues at closed points are $p$-plain (see \ref{fields} for the notation).
	\end{defi}

	We can compare two $\KK$-rational coefficient objects with different fields of scalars looking at their Frobenius characteristic polynomials.
	
	\begin{defi}\label{comp:d}
		Let $\calE_0$ and $\calF_0$ be two coefficient objects that are $\KK$-rational with respect to $\tau$ and $\tau'$ respectively. We say that $\calE_0$ and $\calF_0$ are \textit{$\KK$-compatible} if their characteristic polynomials at closed points are the same as polynomials in $\KK[t]$, after the identifications given by $\tau$ and $\tau'$.
	\end{defi}
	Our general aim in our article will be to convert the numerical data provided by the Frobenius characteristic polynomials at closed points to structural properties of the coefficient objects. As an example, we prove the following general statement for Weil lisse sheaves.

	\begin{prop}\label{p-plain-etale:p}
		Let $X_0$ be a connected variety over $\Fq$ and let $\calV_0$ be a Weil lisse $E_\lambda$-sheaf over $X_0$. If all the eigenvalues of the Frobenius at some closed point $x_0$ of $X_0$ are $\ell$-adic units, then $\calV_0$ is an étale lisse sheaf. In particular, an extension of étale lisse sheaves in $\Weil(X_0,\El)$ is étale.
	\end{prop}
	\begin{proof}The condition on the eigenvalues is preserved after an extension of the base field, thus we may assume that $x_0$ is a rational point. Let $\rho_0$ be the $E_\lambda$-linear representation of $W(X_0,x)$ associated to $\calV_0$ with $x$ an $\F$-point over $x_0$. Write $\Pi_0\subseteq \GL(\calV_{x})$ for the image of $\rho_0$ and $\Pi$ for the image of $\pi_1^{\et}(X,x)$ via $\rho_0$. The group $\Pi_0$ is generated by $\Pi$ and $\rho_0(\gamma)$, where $\gamma$ is some element in $F_{x_0}$. Write $\Gamma$ for the closure in $\GL(\calV_{x})$ of the group generated by $\rho_0(\gamma)$. The subgroup $\Pi\subseteq \GL(\calV_{x})$ is profinite because it is a quotient of the profinite group $\pi_1^{\et}(X,x)$. We want to prove that $\Gamma$ is profinite as well. 
		
	Fix a basis $\{v_1,\dots, v_r\}$ of $\calV_x$ such that $\rho_0(\gamma)$ admits a Jordan normal form. Since the eigenvalues of $\rho_0(\gamma)$ are $\ell$-adic units, $\rho_0(\gamma)$ lies in the closed profinite subgroup $\GL_{r}(\calO_{E_\lambda})\subseteq \GL_{r}({E_\lambda})$, where $\calO_{E_\lambda}$ is the ring of integers of $E_\lambda$. This implies that $\Gamma$ is a closed subgroup of the profinite group $\GL_{r}(\calO_{E_\lambda})$, and therefore it is a profinite group, as we wanted. To conclude the proof note that $\Gamma$ normalizes $\Pi$, so that $\Pi\cdot \Gamma\subseteq \GL(\calV_{x})$ is a profinite subgroup. By construction, $\Pi_0$ is contained in $\Pi\cdot \Gamma$. As a result, the $\ell$-adic representation $\rho_0$ factors through the profinite completion of $W(X_0,x)$, which is $\pi_1^{\et}(X_0,x)$. This concludes the proof.
	\end{proof}

	\subsection{Monodromy groups}
	\label{monodromy-groups-ss}We introduce now the main characters of the article: the \textit{fundamental groups} and the \textit{monodromy groups} of coefficient objects. We will present in Proposition \ref{fun-exact-seq:p} some fundamental exact sequences for these groups. The sequences in Proposition \ref{fun-exact-seq:p} represent the analogue of the well-known exact sequence which relates the geometric and the arithmetic étale fundamental group of a variety. 
	
	\subsubsection{}
	\label{fundamental-groups:d}
	Let $(X_0,x)$ be a smooth connected pointed variety. For every étale admissible field $\El$ we consider the fibre functor $$\omega_{x,\El}: \Weil(X_0,\El)\to \VVec_{\El}$$ attached to $x$. This functor sends a lisse sheaf $\calV_0$ to its stalk $\calV_{x}$. When $\El$ is a $p$-adic admissible field, we defined in §\ref{(x0)} a fibre functor for $\oi(X_0/\Qqs)_{\El}$ over $\Elx$, denoted by $\omega_{x,\El}$. As usual, in order to uniformise the notation, when $\El$ is an étale admissible field we write $\Elx$ for $\El$. Therefore, for every admissible field $\El$, we have a fibre functor $\omega_{x,\El}$ of $\Coef(X^{(x)},\El)$ over $\Elx$. We will denote with the same symbol the fibre functor induced on $\Coef(X_0,\El)$. As the fibre functors commute with the extension of scalars, for every $\ell$ we also have a fibre functor over $\Qlbar$ for $\Qlbar$-coefficient objects. We denote it by $\omega_{x,\Qlbar}$.
	
	\begin{defi}[Fundamental groups] Let $(X_0,x)$ be a smooth connected pointed variety and $\El$ an admissible field. We denote by $\pil(X_0,x)$ the Tannakian group over $\Elx$ of $\Coef(X_0,\El)$ with respect to $\omega_{x,\El}$. We also write $\pil(X,x)$ for the Tannakian group of $\Coef(X^{(x)},\El)$ with respect to the restriction of $\omega_{x,\El}$. The functor $$\Psi_{x,\El}: \Coef(X_0,\El)\to \Coef(X^{(x)},\El)$$ induces a closed immersion $\pil(X,x)\hookrightarrow \pil(X_0,x)$. We denote by $\pil(X_0,x)^{\cst}$ the quotient of $\pil(X_0,x)$ corresponding to the inclusion of $\Coef_{\cst}(X_0,\El)$ in $\Coef(X_0,\El)$.
	\end{defi}
	
	\begin{rema}
		\label{isomorphism-fibre-funct:r}
		Suppose that $\Elx=\El$, then there exists an isomorphism of functors $\eta:\omega_{x,\El} \Rightarrow \omega_{x,\El}\circ (F^s)^*$. For lisse sheaves, this is induced by the choice of an étale path between $x$ and the $\F$-point sent to $x$ via $F^{s}:X\to X$. In the case of overconvergent $F$-isocrystals, it is constructed in \cite[§2.4.18]{Abe}. Thanks to the existence of $\eta$ and Proposition \ref{F-equiv:p}, one can define a \textit{Weil group} for coefficient objects over the field $\El$ (cf. §\ref{a-Weil:c}). 
	\end{rema}

	Every $\El$-coefficient object $\calE_0$ generates three $\El$-linear Tannakian categories, the \textit{arithmetic} one $\langle \calE_0 \rangle\subseteq \Coef(X_0,\El)$, the \textit{geometric} one $\langle \calE \rangle \subseteq \Coef(X^{(x)},\El)$ and the Tannakian category of constant objects $\langle \calE_0 \rangle_{\cst}\subseteq \langle \calE_0 \rangle$. We will consider these categories endowed with the fibre functors obtained by restricting $\omega_{x,\El}$. 
	
	\begin{defi}(Monodromy groups)
		\label{monodromy-groups-d}
			Let $(X_0,x)$ be a smooth connected pointed variety. We denote by $G(\calE_0,x)$ the \textit{(arithmetic) monodromy group} of $\calE_0$, namely the Tannakian group of $\langle \calE_0 \rangle$. The \textit{geometric monodromy group} of $\calE_0$ will be instead the Tannakian group of $\langle \calE \rangle$ and it will denoted by $G(\calE,x)$. We will also consider the quotient $G(\calE_0,x)\twoheadrightarrow G(\calE_0,x)^{\cst}$, which corresponds to the inclusion $\langle \calE_0 \rangle_{\cst}\subseteq \langle \calE_0 \rangle$. These three groups are quotients of the fundamental groups defined in §\ref{fundamental-groups:d}.
	\end{defi}
	
	\begin{rema}When $\calV_0$ is a lisse sheaf and $\rho_0:W(X_0,x)\to \GL(\calV_x)$ is the associated $\ell$-adic representation, then $G(\calV_0,x)$ is the Zariski-closure of the image of $\rho_0$ and $G(\calV,x)$ is the Zariski-closure of $\rho_{0}(\pi_1^{\et}(X,x))$.
		When $\calM_0$ is an overconvergent $F$-isocrystal and $x_0$ is a rational point, $G(\calM,x)$ is the same group as the one defined by Crew in \cite{CrewMon} and denoted by $\DGal(\calM,x)$. This group is also isomorphic to the group $\DGal(\calM,x)$ which appears in \cite{AE}.
	\end{rema}
	
	\begin{rema}
		\label{base-point-r}
		Since $X_0$ is connected, the étale fundamental groups associated to two different $\F$-points of $X_0$ are (non-canonically) isomorphic. Hence, in the case of lisse sheaves, the isomorphism class of the monodromy groups does not depend on the choice of $x$. For overconvergent $F$-isocrystals, by \cite[Theorem 3.2]{DM}, the monodromy groups associated to two different $\F$-points become isomorphic after passing to a finite extension of the field of scalars. We do not know any better result in this case. Note that thanks to \cite{Del11}, we also know that if $\lambda$ is a $p$-adic place, the isomorphism class of $\pil(X_0,x)\otimes_{\Elx} \Qpbar$ is independent of the choice of $x$.
	\end{rema}
	
	Let us present now the fundamental exact sequence for coefficient objects over $X_0$. For overconvergent $F$-isocrystals the sequence is a generalization of the one proven in \cite[Proposition~4.7]{Pal}.

	\begin{prop}\
		\label{fun-exact-seq:p}
		Let $(X_0,x)$ be a smooth pointed variety over $\Fq$, geometrically connected over $\Fqs$, and let $\lambda$ be an admissible place of a number field $E$. 
		
		\begin{itemize}
			\item[{(i)}]
			The natural morphisms previously presented give an exact sequence
			
			\begin{equation*}
			\label{fundamental-exact-sequence}
			1 \to \pil(X,x) \to \pil(X_0,x) \to \pil(X_0,x)^{\cst} \to 1.
			\end{equation*}
			
			\item[{(ii)}]
			For every $E_\lambda$-coefficient object $\calE_0$ and every $\calF\in \langle\calE\rangle$, there exists $\calG_0\in \langle \calE_0\rangle$ such that $\calF\subseteq \calG$.
			\item[{(iii)}]
			For every $\El$-coefficient object $\calE_0$, the exact sequence of (i) sits into the following commutative diagram with exact rows and surjective vertical arrows
			\begin{center}
				\begin{tikzcd}
					1 \arrow{r}& \pil(X,x) \arrow{r}\arrow[d,two heads] &\pil(X_0,x) \arrow{r}\arrow[d,two heads]& \pil(X_0,x)^{\cst} \arrow{r}\arrow[d,two heads]& 1\\
					1\arrow{r} & G(\calE,x) \arrow{r} & G(\calE_0,x)  \arrow{r} & G(\calE_0,x)^{\cst}\arrow{r} & 1.\
				\end{tikzcd}
			\end{center}
			
			\item[{(iv)}] The affine group scheme $\pi_1(\bfC_0,\omega_0)^{\cst}$ is isomorphic to the pro-algebraic completion of $\Z$ over $\KK$ and $G(\calE_0,x)^{\cst}$ is a commutative algebraic group.
			\item[{(v)}]  The affine group scheme $\pil(X_0,x)^{\cst}$ is canonically isomorphic to $\pil(\Spec(\Fqs),x).$ In particular, the profinite group $\pi_0(\pil(X_0,x)^{\cst})$ is canonically isomorphic to $\Gal(\F/\Fqs)$.
		\end{itemize}
		
	\end{prop}
	
	\begin{proof}
		By Proposition \ref{F-equiv:p}, the category $\Coef(X^{(x)}, \Elx)$ endowed with the endofunctor $(F^s)^*$ is a neutral Tannakian category with Frobenius, in the sense of Definition \ref{a-defi}. Thus by Theorem \ref{a-fundamental-exact-sequence:t} we get all the parts from (i) to (iv).
		
	Passing to (v), write $q_{X_0}: X_0\to \Spec(\Fqs)$ for the morphism induced by the $\F$-point $x$. The inverse image functor $$q_{X_0}^*:\Coef(\Spec(\Fqs),\El^{(x_0)})\to\Coef_{\cst}(X_0,\El^{(x_0)})$$ admits as a quasi-inverse the functor $q_{X_0*}$ which sends a constant coefficient object $\calE_0$ to the vector space $H^0(X^{(x)},\calE)$ endowed with the action of $(F^s)^*$. This implies that $q_{X_0}^*$ induces an isomorphism $$\pil(X_0,x)^{\cst}\iso \pil(\Spec(\Fqs),x).$$ Since $\Coef(\Spec(\Fqs),\El^{(x_0)})$ is canonically equivalent to $\Rep_{\El^{(x_0)}}(W(\F/\Fqs))$, the profinite group $\pi_0(\pil(X_0,x)^{\cst})$ is canonically isomorphic to $\Gal(\F/\Fqs)$.
	\end{proof}

	\subsection{Comparison with the étale fundamental group}\label{comparison-etale:ss}
	\subsubsection{}
	We continue our analysis of the fundamental groups of coefficient objects. Here we focus our attention on the group of connected components. The statements of this section are fairly easy for lisse sheaves and more difficult for overconvergent $F$-isocrystals. In the latter case, Crew have already studied the problem when $X_0$ is a smooth curve, \cite{CrewMon}. Later in \cite{Ete}, \'Etesse proved that overconvergent isocrystals (with and without Frobenius structure) over smooth varieties of arbitrary dimension satisfy étale descent\footnote{In the article he states the result for overconvergent $F$-isocrystals, but the same proof works without Frobenius structure.}. This allows a generalization of Crew's work. 

	Drinfeld and Kedlaya presented in \cite[Appendix B]{DK} how to perform such a generalization for the arithmetic fundamental group of overconvergent $F$-isocrystals. We will be mainly interested in the extension of their result to the geometric fundamental group.
	
	\subsubsection{} Let $(X_0,x)$ be a smooth connected pointed variety over $\Fq$ and $\El$ an admissible field for $X_0$. Following \cite[Remark B.2.5]{DK}, we define $$\Rep_{\El}^{\mathrm{smooth}}(\pi_1^{\et}(X_0,x)):=2\textrm{-}\varinjlim_{H} \Rep_{\El}(\pi_1^{\et}(X_0,x)/H) $$ where $H$ varies among the normal open subgroups of $\pi_1^{\et}(X_0,x)$. This category is naturally endowed with a fully faithful embedding $$\Rep_{\El}^{\mathrm{smooth}}(\pi_1^{\et}(X_0,x))\hookrightarrow \Coef(X_0,\El).$$ The essential image is closed under subquotients. Therefore, this functor induces a surjective morphism $\pil(X_0,x)\twoheadrightarrow \pie(X_0,x)$, where $\pie(X_0,x)$ denotes here the pro-constant profinite group scheme over $\Elx$ associated to the étale fundamental group of $X_0$. The subcategory $$\Rep_{\El}^{\mathrm{smooth}}(\Gal(\F/k_{X_0}))\subseteq \Rep_{\El}^{\mathrm{smooth}}(\pi_1^{\et}(X_0,x))$$ of representations which factor through $\Gal(\F/k_{X_0})$ is sent by the previous functor to the category of constant coefficient objects. This implies that the composition of the morphisms $$\pil(X_0,x)\twoheadrightarrow \pie(X_0,x)\twoheadrightarrow\Gal(\F/k_{X_0})$$ factors through $\pil(X_0,x)^{\cst}$. By Proposition \ref{fun-exact-seq:p}.(v), the induced morphism $\pil(X_0,x)^{\cst}\twoheadrightarrow \Gal(\F/k_{X_0})$ is surjective with connected Kernel. Finally, the homotopy exact sequence for the étale fundamental group and the fundamental exact sequence of Proposition \ref{fun-exact-seq:p}.(i) fit in a commutative diagram
	\begin{equation}
	\label{square-pie}
	\begin{tikzcd}
	1\arrow[r]&\pil(X,x) \arrow[r]\arrow[d] & \pil(X_0,x)  \arrow[d, two heads] \arrow[r] & \pil(X_0,x)^{\cst} \arrow[r]\arrow[d, two heads] & 1 \\
	1\arrow[r]&\pi_1^{\et}(X,x) \arrow[r] & \pi_1^{\et}(X_0,x)\arrow[r] & \Gal(\F/k_{X_0})\arrow[r]&1.\
	\end{tikzcd}
	\end{equation}
	The central and the right vertical arrows are the morphisms previously constructed. The left one is the unique morphism making the diagram commutative.

	\begin{prop}
		\label{connected-components-etale-proposition}
		Let $(X_0,x)$ be a smooth connected pointed variety over $\Fq$. For every admissible place $\lambda$ we have a commutative diagram
		\begin{equation}\label{pi0-diagram}
		\begin{tikzcd}
		1 \arrow{r}& \pi_0(\pil(X,x)) \arrow[r]\arrow[d,"\sim" labl,"\varphi"] &\pi_0( \pil(X_0,x)) \arrow{r}\arrow[d,"\sim" labl,"\varphi_0"]& \pi_0(\pil(X_0,x)^{\cst}) \arrow{r}\arrow[d,"\sim" labl,"\varphi^{\cst}_0"]& 1 \\
		1\arrow[r]&\pi_1^{\et}(X,x) \arrow[r] & \pi_1^{\et}(X_0,x)\arrow[r] & \Gal(\F/k_{X_0})\arrow[r]&1,\	
		\end{tikzcd}
		\end{equation}
		where the vertical arrows are isomorphisms and the rows are exact. The diagram is functorial in $(X_0,x)$ when it varies among the smooth connected pointed varieties over $\Fq$.
	\end{prop}
	
	\begin{proof}
		The diagram is constructed applying the functor $\pi_0$ to (\ref{square-pie}), hence it is functorial. To prove that it has all the desired properties we may extend the field of scalars to $\Qlbar$. We start by showing that the upper row is exact. As the functor $\pi_0$ is right exact, it is enough to prove the injectivity of the morphism $\pi_0(\pil(X,x))\to\pi_0( \pil(X_0,x))$. The $\pi_0$ of the Tannakian group of a Tannakian category is the Tannakian group of the subcategory of finite objects. Thus we have to prove that for every irreducible finite geometric $\Qlbar$-coefficient object $\calE$, there exists a finite object $\calF_0 \in \Coef(X_0,\Qlbar)$, such that $\calE$ is a subquotient of $\calF$. 
		
		By Lemma \ref{a-normal:l}, there exists $\calF'_0\in \Coef(X_0,\Qlbar)$ such that $\calE$ is a subobject of $\calF'$. As $\calE$ is irreducible, we can even assume $\calF'_0$ to be irreducible. In particular, there exist $g_1,\dots,g_n\in G(\calF'_0,x)(\Qlbar)$ such that $\omega_{x,\Qlbar}(\calF'_0)=\sum_{i=1}^n g_i (\omega_{x,\Qlbar}(\calE))$. The algebraic group $G(\calF',x)$ is normal in $G(\calF'_0,x)$ by Proposition \ref{fun-exact-seq:p}, thus the vector spaces $g_i (\omega_{x,\Qlbar}(\calE))$ are $G(\calF',x)$-stable for every $i$. In addition, their monodromy groups as representations of $G(\calF',x)$ are all finite, as they are conjugated to the monodromy group of $\calE$. Therefore $\calF'$, being a sum of finite objects, is a finite object. 
		
		Let $W(\calF'_0,x)$ be the Weil group of $\langle \calF'\rangle$, as in §\ref{a-Weil:c}. Since $G(\calF',x)$ is finite, there exists $n\in \Z_{>0}$ such that $(F^n)^*$ acts trivially on it. If $\rho'$ is the representation of $G(\calF',x)$ associated to $\calF'$, then $(F^n)^*\rho'=\rho'$. Thus $\rho:=\bigoplus_{i=0}^{n-1} (F^i)^*\rho'$ can be endowed with a Frobenius structure $$\Phi:F^*\left(\bigoplus_{i=0}^{n-1} (F^i)^*\rho'\right)\iso \bigoplus_{i=0}^{n-1} (F^i)^*\rho'$$ such that, for every $1\leq i\leq n-1$, the restriction of $\Phi$ to $F^*\left((F^i)^*\rho'\right)$ is the canonical isomorphism $F^*\left((F^i)^*\rho'\right)\iso (F^{i+1})^*\rho'$. The pair $(\rho,\Phi)$ induces a representation of $W(\calF'_0,x)$ with finite image and thus a finite coefficient object $\calF_0$. The original geometric coefficient object $\calE$ is a subobject of $\calF$, therefore $\calF_0$ satisfies the properties that we wanted.
		
		Finally, we prove that the vertical arrows of (\ref{pi0-diagram}) are isomorphisms.
		The morphism $\varphi_0^{\cst}$ is an isomorphism by Proposition~\ref{fun-exact-seq:p}.(v). By diagram chasing, it remains to prove that $\varphi_0$ is an isomorphism. For lisse sheaves, this is quite immediate. If a lisse sheaf has finite arithmetic monodromy group, its associate $\ell$-adic representation factors through a finite quotient of the Weil group of $X_0$. In the $p$-adic case one can prove that $\varphi_0$ is an isomorphism using \cite[Theorem~2.3.7]{KedUR}, as it is explained in \cite[Proposition~B.7.6.(i)]{DK}.
	\end{proof}
	
	\begin{prop}
		\label{connected-after-cover-proposition}
		Let $\calE_0$ be a $\Qlbar$-coefficient object over $(X_0,x)$.
		\begin{itemize}
			\item[{(i)}] For every finite étale morphism $f_0:(Y_0,y)\to(X_0,x)$ of pointed varieties, the natural morphisms $G(f_0^*\calE_0,y)\to G(\calE_0,x)$ and $G(f^*\calE,y)\to G(\calE,x)$ are open immersions.
			\item[{(ii)}] There exists a choice of $f_0:(Y_0,y)\to(X_0,x)$ such that $G(f_0^*\calE_0,y)\iso G(\calE_0,x)^\circ$ and $G(f^*\calE,y)\iso G(\calE,x)^\circ$.  
		\end{itemize}
		
	\end{prop}
	
	\begin{proof}
		We note that by Proposition~\ref{connected-components-etale-proposition} the group of connected components of the arithmetic monodromy group (resp. geometric monodromy group) are quotients of the arithmetic étale fundamental group (resp. geometric étale fundamental group), thus (i) implies (ii). 
		
		When $\calE_0$ is a lisse sheaf, (i) is well-known. If $\calE_0$ is an overconvergent $F$-isocrystal, the result on the arithmetic monodromy groups is a consequence of \cite[Proposition~B.7.6.(ii)]{DK}. It remains to prove (i) for the geometric monodromy groups of overconvergent $F$-isocrystals. It is enough to treat the case when $Y_0\to X_0$ is a Galois cover with Galois group $H$ and $Y_0$ is geometrically connected over $\Fq$. As $Y_0$ is geometrically connected over $\Fq$, the group $H$ acts on $\langle f^*\calE \rangle$ via $\Qpbar$-linear autoequivalences. Let $\langle f^*\calE \rangle^H$ be the category of $H$-equivariant objects in $\langle f^*\calE \rangle$. We choose isomorphisms of fibre functors between $\omega_{y,\Qpbar}$ and $\omega_{h(y),\Qpbar}$ for every $h\in H$. This choice induces an action of $H$ on $G(f^*\calE,y).$
		
		By \cite{Ete}, overconvergent isocrystals with and without $F$-structure satisfy étale descent. Therefore, there exist fully faithful embeddings $\langle \calE \rangle\hookrightarrow \langle f^*\calE \rangle^H$ and $\langle f^*\calE \rangle^H\hookrightarrow \oi(X_0/\Qq)_{\Qpbar}$. The former embedding induces a morphism at the level of the Tannakian groups $\varphi:G(f^*\calE,y)\rtimes H\to G(\calE,x)$. By definition, the subcategory $\langle \calE \rangle\subseteq \oi(X_0/\Qq)_{\Qpbar}$ is closed under the operation of taking subquotients. Thus, the same is true for $\langle\calE\rangle \subseteq \langle f^*\calE \rangle^H$. This proves that $\varphi$ is surjective, which in turn implies that $G(f^*\calE,y)$ is an open subgroup of $G(\calE,x)$.
	\end{proof}

	\subsection{Rank 1 coefficient objects}
	\label{rank1:ss}This section is an interlude on rank $1$ coefficient objects. One of the starting points of Weil II is a finiteness result for rank 1 lisse sheaves, consequence of class field theory. Thanks to a reduction to unramified $p$-adic representations of the étale fundamental group, the same statement is now known for overconvergent $F$-isocrystals of rank $1$.
	
	\begin{theo}[{{\cite[Proposition~1.3.4]{Weil2}, \cite[Lemma 6.1]{Abe2}}}]
		\label{rank-1-finite-order:t} Let $X_0$ be a smooth variety over $\Fq$. Every $\El$-coefficient object of rank $1$ is a twist of a finite $\El$-coefficient object.
	\end{theo}
	
	\begin{coro}
		\label{rank-1-finite-order:c}
		For every $\Qlbar$-coefficient object $\calE_0$ over $X_0$, there exist a positive integer $n$ and elements $a_1,\dots,a_n\in \Qlbar^{\times}$ such that $$\calE_0^{\sms}\simeq \bigoplus_{i=1}^n \calF_{i,0}^{(a_i)},$$ where for each $i$ the coefficient object $\calF_{i,0}$ is irreducible with finite order determinant. The elements $a_1,\dots,a_n$ are uniquely determined up to permutation and multiplication by a root of unity. If $\calE_0$ is $E$-rational, the elements $a_1,\dots,a_n$ can be chosen so that $a_i^{r_i}\in E$ for every $i$, where $r_i$ is the rank of $\calF_{i,0}$. 
	\end{coro}
	
	Corollary \ref{rank-1-finite-order:c} is important as it allows to reduce many statements on coefficient objects to the case of absolutely irreducible coefficient objects with finite order determinant. It is convenient to introduce the following definitions.
	
	\begin{defi}[Twist classes]\label{t-c:d}
		For a prime $\ell$, we denote by $\Theta_\ell$ the torsion-free abelian group $\Qlbar^{\times}/\mu_{\infty}(\Qlbar)$. We say that an element in $\Theta_\ell$ is an \textit{($\ell$-adic) twist class}. We say the class of $1$ in $\Theta_\ell$ is the \textit{trivial twist class}. Let $\calE_0$ be a $\Qlbar$-coefficient object and let $a_1,\dots,a_n\in \Qlbar^{\times}$ be elements as in Corollary \ref{rank-1-finite-order:c}. We say that the classes $[a_1],\dots,[a_n]$ in $\Theta_\ell$ are the \textit{twist classes of $\calE_0$} and we denote by $\Theta(\calE_0)$ the subset $\{[a_1],\dots,[a_n]\}\subseteq \Theta_\ell$. We write $\mathbb{X}(\calE_0)$ for the group generated by $\Theta(\calE_0)$ in $\Theta_\ell$ and by $\mathbb{X}(\calE_0)_\Q$ the $\Q$-linear subspace $\mathbb{X}(\calE_0)\otimes_\Z \Q\subseteq \Theta_\ell\otimes_\Z \Q$. If the only twist class of $\calE_0$ is the trivial one, we say that $\calE_0$ is \textit{untwisted}. 
	\end{defi}

	An important application of Theorem \ref{rank-1-finite-order:t} is the \textit{global monodromy theorem}. In this case, the extension to overconvergent $F$-isocrystals is due to Crew.
	
	\begin{theo}[Grothendieck, Crew]
		\label{global-monodromy-theorem}
		For every coefficient object $\calE_0$, the radical\footnote{For us, the radical of an algebraic group is defined to be its maximal normal solvable connected subgroup.} of $G(\calE,x)$ is unipotent.
	\end{theo}
	
	\begin{proof}
		In the case of lisse sheaves, this is a theorem of Grothendieck, and it is proven in \cite[Théorème 1.3.8]{Weil2}. In the $p$-adic case, Crew has proven the result when $X_0$ is a smooth curve \cite[Theorem~4.9]{CrewMon}. One obtains the result in higher dimensions replacing [\textit{ibid.}, Proposition 4.6] with Proposition~\ref{connected-after-cover-proposition} and [\textit{ibid.}, Corollary~1.5] with Theorem~\ref{rank-1-finite-order:t}.
		
	\end{proof}

\begin{coro}\label{derived-sub:c}
	Let $\calE_0$ be a geometrically semi-simple coefficient object. The neutral component $G(\calE,x)^\circ$ is a semi-simple algebraic group which coincides with the derived subgroup of $G(\calE_0,x)^\circ$. 
\end{coro}

\begin{proof}
Thanks to Theorem~\ref{global-monodromy-theorem}, $G(\calE,x)^\circ$ is a semi-simple algebraic group, therefore $$G(\calE,x)^\circ=[G(\calE,x)^\circ,G(\calE,x)^\circ]\subseteq \left[G(\calE_0,x)^\circ,G(\calE_0,x)^\circ\right].$$
	On the other hand, by Proposition \ref{fun-exact-seq:p}.(iv), the quotient $G(\calE_0,x)^\circ/G(\calE,x)^\circ$ is commutative, which implies that $$\left[G(\calE_0,x)^\circ,G(\calE_0,x)^\circ\right]\subseteq G(\calE,x)^\circ.$$ This concludes the proof.
\end{proof}

	One can even enhance Theorem \ref{global-monodromy-theorem} thanks to the following lemma.
	
	\begin{lemm}
		\label{determinant:l}
		For every coefficient object $\calE_0$ and every connected finite étale cover $f_0:(Y_0,y)\to (X_0,x)$, we have $\Theta(\calE_0)=\Theta(f_0^*\calE_0)$.
	\end{lemm}
	\begin{proof}
		After taking semi-simplification and twists, we may assume that $\calE_0$ is absolutely irreducible with finite order determinant. In this case, the result is proven in \cite[Proposition 3.6.1]{Dri2} for lisse sheaves. The proof is the same for overconvergent $F$-isocrystals, as they satisfy étale descent by \cite{Ete}.
	\end{proof}
	
	\begin{theo}\label{strong-gmt:t}
		
		Let $\calE_0$ be a $\Qlbar$-coefficient object. The following properties are equivalent.
		\begin{itemize}
			\item[{(i)}] The neutral component of $G(\calE_0,x)/G(\calE,x)$ is unipotent.
			\item[{(ii)}] The radical of $G(\calE_0,x)$ is unipotent.
			
			\item[{(iii)}] The coefficient object $\calE_0$ is untwisted.
		\end{itemize}
		In particular, untwisted coefficient objects form a Tannakian subcategory of $\Coef(X_0,\Qlbar)$.
	\end{theo}
	
	\begin{proof} The implication (i) $\Rightarrow$ (ii) follows from Theorem \ref{global-monodromy-theorem} and the other direction follows from the fact that $G(\calE_0,x)/G(\calE,x)$ is a commutative quotient of $G(\calE_0,x)$. Let us show now that (ii) and (iii) are equivalent as well. If $\calE_0$ is a coefficient object satisfying (ii), all the rank $1$ coefficient objects in $\langle\calE_0\rangle$ have finite order under tensor. In particular, if $\calF_0$ is an irreducible subquotient of $\calE_0$, its determinant has finite order under tensor. This implies that $\calE_0$ is untwisted. Conversely, let us assume that $\calE_0$ is untwisted. Thanks to Proposition \ref{connected-after-cover-proposition}, there exists a connected finite étale cover $f_0:(Y_0,y)\to(X_0,x)$ such that $G(f_0^*\calE_0,y)=G(\calE_0,x)^\circ$. By Lemma \ref{determinant:l}, the inverse image $f_0^*\calE_0$ remains untwisted. This shows that we may assume that $G(\calE_0,x)$ is connected. Note that it is also harmless to assume that $\calE_0$ is semi-simple by passing to the semi-simplifications. We are reduced to showing that the centre $Z$ of $G(\calE_0,x)$ is finite. To prove this we may further assume that $\calE_0$ is irreducible. Indeed, if $\calE_0=\calF_0\oplus\calG_0$ and $Z_1$ and $Z_2$ are the centres of $G(\calF_0,x)$ and $G(\calG_0,x)$ respectively, then $Z\subseteq Z_1\times Z_2$. Therefore, if $Z_1$ and $Z_2$ are finite the same holds for $Z$.

    If we assume that $\calE_0$ is irreducible, the representation of $Z$ on $\omega_{x,\Qlbar}(\calE_0)$ decomposes into a direct sum $\chi^{\oplus r}$ where $r$ is the rank of $\calE_0$ and $\chi$ is a character of $Z$. By construction, this representation is faithful, therefore $\chi$ generates the group of characters of $Z$. On the other hand, we know that $\chi$ is finite because $\calE_0$ is untwisted. This shows that $Z$ is a finite group scheme, as we wanted.
	\end{proof}

	\begin{coro}\label{characters-twistc:c}
		Let $\calE_0$ be a $\Qlbar$-coefficient object.
		\begin{itemize}
			\item[\normalfont (i)] 	For every $\calF_0\in \langle\calE_0\rangle$, we have $\Theta(\calF_0)\subseteq\X(\calE_0)$.
			\item[\normalfont (ii)] The map $X^*(G(\calE_0,x))\to \X(\calE_0)$ which associates to a rank $1$ coefficient object its twist class has finite kernel and cokernel.
		\end{itemize}
		
	\end{coro}
	\begin{proof}
		It is enough to prove (i) for $\calF_0=\calE_0^{\otimes m}\otimes(\calE_0^{\vee})^{\otimes n}$ with $m,n\in \N$. In addition, we may assume that $\calE_0$ admits a unique twist class, hence it can be written as $\calF_0^{(a)}$ with $\calF_0$ untwisted. By Theorem \ref{strong-gmt:t}, the coefficient object $\calF_0^{\otimes m}\otimes(\calF_0^{\vee})^{\otimes n}$ is untwisted. Therefore, $[a^{m-n}]\in \X(\calE_0)$ is the unique twist class of $\calE_0^{\otimes m}\otimes(\calE_0^{\vee})^{\otimes n}$. This proves part (i). For part (ii) we first note that the kernel is finite because untwisted rank $1$ coefficient objects have finite order under tensor. To prove that the cokernel is finite as well it is enough to prove that for every twist class $[a]$ of $\calE_0$, there exists a rank $1$ coefficient object $\calL_0\in \langle\calE_0\rangle$ with twist class $[a^n]$ for some $n\geq 1$. Since, by definition, there exists an irreducible object $\calF_0\in \langle\calE_0\rangle$ with twist class $[a]$, we can pick $\calL_0:=\det(\calF_0)$. 
	\end{proof}

	\subsection{Weights}
	\label{Weights-ss} 
	
	In Weil II Deligne introduced the \textit{theory of weights} for lisse sheaves. The same theory is now available for overconvergent $F$-isocrystals, thanks to the work of Kedlaya in \cite{KedWeil2}. Here the main theorem. 
	
	\begin{theo}[Deligne, Kedlaya]
		\label{theory-of-weights-theorem}
		Let $X_0$ be a smooth geometrically connected variety over $\Fq$ and $\calE_0$ a $\iota$-mixed coefficient object over $X_0$ of $\iota$-weights $\leq w$. If $\alpha$ is an eigenvalue of $F$ acting on $H^n_c(X,\calE)$, then $|\iota(\alpha)|\leq q^{(w+n)/2}$.
	\end{theo}
	\begin{proof}
		For lisse sheaves this is \cite[Corollaire 3.3.5]{Weil2}. For overconvergent $F$-isocrystals it is proven by Kedlaya in \cite[Theorem 6.6.2]{KedWeil2}.
	\end{proof}
	
%	\begin{coro}[]
	%	\label{split:c}
	%	Let $$0\to \calF_0 \to \calE_0 \to \calG_0	\to 0$$
	%	be an exact sequence of coefficient objects such that $\calF_0$ and $\calG_0$ are $\iota$-pure of weights $w_1$ and $w_2$ respectively.
	%	\begin{itemize}
	%		\item[{(i)}] If $w_1>w_2-1$ the sequence splits geometrically.
	%		\item[{(ii)}]If $w_1\neq w_2$ and the sequence splits geometrically, then it splits.
	%	\end{itemize} 
%	\end{coro}
	
	\begin{coro}
		\label{filt-geo-semi:c}Let $X_0$ be a smooth variety. The following statements are true.
		\begin{itemize}
		\item[{\normalfont(i)}]For every $\iota$-mixed coefficient object there exists an increasing filtration $$0=W_{-1}(\calE_0)\subsetneq W_0(\calE_0)\subsetneq\cdots \subsetneq  W_n(\calE_0)=\calE_0$$ where for every $0\leq i\leq n$, the quotient $W_i(\calE_0)/W_{i-1}(\calE_0)$ is $\iota$-pure of weight $w_i$ and $w_0<w_1<\dots<w_n$.
		\item[{\normalfont(ii)}]	Every $\iota$-pure coefficient object is geometrically semi-simple. Conversely, every $\iota$-mixed geometrically semi-simple coefficient object is a direct sum of $\iota$-pure coefficient objects.
		\end{itemize}
		
	\end{coro}
\begin{proof}
This follows from Theorem \ref{theory-of-weights-theorem} as proved in {\cite[Théorème 3.4.1]{Weil2}}.
\end{proof}

	For every $\Qlbar$-coefficient object $\calE_0$ on $X_0$, we can put together all the characteristic polynomials at closed points and form a formal series $$L_{X_0}(\calE_0,t):=\prod_{x_0\in |X_0|}P_{x_0}(\calE_0,t^{\deg(x_0)})^{-1}\in \Qlbar[[t]].$$ This is called the \textit{$L$-function} of $\calE_0$.

	\begin{theo}[Trace formula]
		\label{Trace-formula-theorem}
		If $X_0$ is geometrically connected over $\Fq$, for every coefficient object $\calE_0$ we have
		$$L_{X_0}(\calE_0,t)=\prod_{i=1}^{2d}\det(1-Ft,H^i_c(X,\calE))^{(-1)^{i+1}}.$$
	\end{theo}
	\begin{proof}
		For lisse sheaves, this is the classical Grothendieck's formula, in the $p$-adic case see \cite[Théorème 6.3]{EteLF}.
	\end{proof}
	Thanks to the theory of weights, this formula can be used to compare the global sections of compatible coefficient objects. The theory of weights is needed to control the possible cancellations between the factors of the numerator and the denominator.
	
	\begin{prop}[{\normalfont\cite[Cor. VI.3]{Laf}, \cite[Prop. 4.3.3]{Abe}}]\label{independence-triviality:p}
		Let $X_0$ be a smooth geometrically connected variety over $\F_q$ of dimension $d$. For every $\iota$-pure coefficient object $\calE_0$ of $\iota$-weight $w$, the dimension of $H^0(X,\mathcal{E})$ is equal to the number of poles of $\iota(L(X_0,\calE^{\vee}_0(d)))$, counted with multiplicity, with absolute value $q^{w/2}$. If we also assume $\calE_0$ to be semi-simple, the dimension of $H^0(X_0,\mathcal{E}_0)$ is equal to the order of the pole of $L(X_0,\calE^{\vee}_0(d))$ at $1$.
	\end{prop}
	
	\begin{coro}\label{independence-triviality:c}
		Let $X_0$ be a smooth geometrically connected pointed variety over $\Fq$. Let $\calE_0$ and $\calF_0$ be $E$-compatible coefficient objects and suppose that $\Qlbar$-coefficient object $\calE_0$ is $\iota$-mixed. The following statements are true.
		\begin{itemize}
			\item[\normalfont(i)] If $\calE_0$ and $\calF_0$ are geometrically semi-simple, then $\dim(H^0(X,\calE))=\dim(H^0(X,\calF))$.
			\item[\normalfont(ii)] If $\calE$ is absolutely irreducible the same is true for $\calF$.
			\item[\normalfont(iii)] If $\calE_0$ and $\calF_0$ are semi-simple, then $\dim(H^0(X_0,\calE_0))=\dim(H^0(X_0,\calF_0))$.			\item[\normalfont(iv)] If $\calE_0$ is absolutely irreducible the same is true for $\calF_0$.
		\end{itemize}
	\end{coro}
	\begin{proof}Let $\Qlpbar$ be the algebraic closure of the field of scalars of $\calF_0$ and let $\iota':\Qlpbar\iso\C$ be an isomorphism which agrees with $\iota$ when restricted to $E$. The $\Qlpbar$-coefficient object $\calF_0$ is $\iota'$-mixed and its $\iota'$-weights are equal to the $\iota$-weights of $\calE_0$. In addition, if $W_*(\calE_0)$ and $W_*(\calF_0)$ are the weight filtrations of Corollary \ref{filt-geo-semi:c}.(i), for each $i$ the quotients $W_i(\calE_0)/W_{i-1}(\calE_0)$ and $W_i(\calF_0)/W_{i-1}(\calF_0)$ are $E$-compatible. In order to prove part (i), it is enough to show that statement for these subquotients. But for them, (i) follows from Proposition \ref{independence-triviality:p}. In addition, applying part (i) to $\End(\calE_0)$ and $\End(\calF_0)$, we also get part (ii). For part (iii) and (iv) one argues similarly.
	\end{proof}

Thanks to the theory of weights one can even prove the following result. 
	
	\begin{prop}[{\cite[A.3]{Abe}}]
		\label{Chebotarev-proposition}
		Two $\iota$-mixed $\Qlbar$-coefficient objects with the same characteristic polynomials at closed points have isomorphic semi-simplifications.
	\end{prop}

	\begin{rema} We will see later that every coefficient object is actually $\iota$-mixed (Corollary~\ref{ever-is-mixe:c}). Therefore, Corollary \ref{independence-triviality:c} and Proposition \ref{Chebotarev-proposition} can be applied to every coefficient object. Note that for lisse sheaves Proposition \ref{Chebotarev-proposition} is classically obtained, without the theory of weights, as a consequence of Chebotarev's density theorem, \cite[Theorem~7]{Cheb}.
	\end{rema}
	
	\subsection{Deligne's conjecture}
	\label{Deligne-conjectures-ss}
	We are ready now to present the state of art of \cite[Conjecture 1.2.10]{Weil2}\footnote{We decided to omit here the part of the conjecture on the $p$-adic valuations of the Frobenius eigenvalues at closed points. Moreover, we will work with smooth varieties, even if the conjecture was originally stated for normal varieties.}. See also \cite{AE} and \cite{Ked} for other overviews. The extension of the statement to overconvergent $F$-isocrystals was firstly proposed by Crew in \cite[Conjecture 4.13]{CrewMon}. This corresponds to the choice of the category of overconvergent $F$-isocrystals as a possible candidate for Deligne's “petits camarades cristallins”. 
	\begin{conj}
		\label{Deligne-conjectures}
		Let $X_0$ be a smooth variety over $\F_q$, let $\ell$ be a prime number and let $\calE_0$ be an irreducible $\Qlbar$-coefficient object whose determinant has finite order. The following statements hold.
		\begin{enumerate}
			\item[{(i)}] $\calE_0$ is pure of weight $0$.
			\item[{(ii)}] There exists a number field $E \subseteq \Qlbar$ such that $\calE_0$ is $E$-rational.
			\item[{(iii)}] $\calE_0$ is $p$-plain.
			\item[{(iv')}] If $E$ is a number field as in (ii), then for every prime $\ell'$ (even $\ell'=\ell$ or $\ell'=p$) and for every inclusion $\tau: E \hookrightarrow \Qlpbar$, there exists an absolutely irreducible $\Qlpbar$-coefficient object, $E$-rational with respect to $\tau$, which is $E$-compatible with $\calE_0$.
		\end{enumerate}
	\end{conj}
Conjecture \ref{Deligne-conjectures}.(iv') is commonly known as the\textit{ companions conjecture}, since compatible coefficient objects are often called \textit{companions}. Deligne originally proposed a stronger variant of (iv') which we shall discuss later in §\ref{compatible-systems-ss}. When $X_0$ is a smooth curve, Conjecture \ref{Deligne-conjectures} is proven using the Langlands program. More precisely, it follows from the Langlands reciprocity conjecture for $\GL_r$ over function fields and the Ramanujan--Petersson conjecture, both proven by L. Lafforgue, and from the $p$-adic analogue of the Langlands reciprocity conjecture for overconvergent $F$-isocrystals, obtained by Abe.
	\begin{theo}[{\normalfont\cite[Théorème VII.6]{Laf}, \cite[§4.4]{Abe}}]
		\label{Langlands-theorem}
		If $X_0$ is a smooth curve, Conjecture \ref{Deligne-conjectures} is true.
	\end{theo}
	
	The extension of the results to higher dimensional varieties is performed via a reduction to curves. One of the key ingredients is a “Lefschetz theorem” for coefficient objects. 
	
	\begin{theo}[Katz, Abe--Esnault]\label{dom-curve-theorem}
		Let $X_0$ be a smooth geometrically connected variety over $\F_q$. For every lisse sheaf $\calE_0$ over $X_0$ and every reduced finite closed subscheme $S_0\subseteq X_0$, there exists a smooth geometrically connected curve $C_0$ and a morphism $f_0:C_0\to X_0$ with a section $S_0\to C_0$, such that the inverse image functor $\langle \calE_0 \rangle \to \langle f_0^*\calE_0 \rangle$ is an equivalence of categories. The same is true when $\calE_0$ is a $\iota$-pure overconvergent $F$-isocrystal.
	\end{theo}
	
	\begin{proof}
		For lisse sheaves see \cite[Lemma 6 and Theorem 8]{Katz} as well as \cite{KatzCor}. In the $p$-adic case see the (proof of) \cite[Theorem~3.10]{AE}.
	\end{proof}
	
	Thanks to Theorem \ref{dom-curve-theorem} and the work of Deligne in \cite{DelFin}, the first three parts of the conjecture follow from the case of curves.
	\begin{theo}[L.~Lafforgue, Abe, Deligne, Abe--Esnault, Kedlaya]
		\label{Deligne-conjecture-general-theorem}
		Parts (i), (ii) and (iii) of Conjecture \ref{Deligne-conjectures} are true for every smooth variety over $\Fq$.
	\end{theo}
	\begin{proof} For lisse sheaves, parts (i) and (iii) follow directly from Theorem \ref{Langlands-theorem}, thanks to Theorem \ref{dom-curve-theorem}. Switching to overconvergent $F$-isocrystals, Conjecture \ref{Deligne-conjectures}.(i) is proven in \cite[Theorem 2.7]{AE} and independently in \cite[Theorem 3.1.9.(i)]{Ked}. Part (iii) then follows from Theorem \ref{Langlands-theorem} thanks to part (i) and Theorem \ref{dom-curve-theorem}. Part (ii) is proven in \cite[Theorem~3.1]{DelFin} for lisse sheaves and in \cite[Lemma 4.1]{AE} and \cite[Theorem~3.4.2]{Ked} for overconvergent $F$-isocrystals.
	\end{proof}
	
	The generalization of part (iv') to higher dimensional varieties is yet incomplete. For the moment we know how to construct from a coefficient object of both types compatible lisse sheaves. In dimension greater than 1, we do not know how to construct, in general, compatible overconvergent $F$-isocrystals.
	
	\begin{theo}[L.~Lafforgue, Abe, Drinfeld, Abe--Esnault, Kedlaya]
		\label{weak-companions-smooth-theorem}
		Let $X_0$ be a smooth variety over $\F_q$ and $E$ a number field. Let $\calE_0$ be an absolutely irreducible $E$-rational coefficient object with finite order determinant over $X_0$. For every prime $\ell$ different from $p$ and every embedding $\tau: E \hookrightarrow \Qlbar$, there exists a $\Qlbar$-coefficient object which is $E$-rational with respect to $\tau$ and $E$-compatible with $\calE_0$.
	\end{theo}
	
	\begin{proof}
		Drinfeld proved the result when $\calE_0$ is a lisse sheaf, \cite{Dri}. The proof uses L. Lafforgue's result and a certain gluing theorem for lisse sheaves [\textit{ibid.}, Theorem~2.5]. The gluing theorem is inspired by the seminal work of Wiesend in \cite{Wie}. When $\calE_0$ is an overconvergent $F$-isocrystal the result was proven in \cite{AE} and later in \cite{Ked}. They both use Drinfeld's gluing theorem for lisse sheaves. In \cite{AE}, the authors prove and use Theorem~\ref{dom-curve-theorem} for overconvergent $F$-isocrystals. Instead in \cite{Ked}, it is proven a weaker form of Theorem~\ref{dom-curve-theorem}, namely [\textit{ibid}., Lemma 3.2.1], which is enough to conclude. 
	\end{proof}
	
	The known parts of Deligne's conjecture have many important consequences. For example, by Corollary \ref{rank-1-finite-order:c}, we have the following result.
	\begin{coro}\label{ever-is-mixe:c}
		Every coefficient object over $X_0$ is $\iota$-mixed.
	\end{coro}

We have already seen in §\ref{Weights-ss} how, using the theory of weights, one can recover some useful information from the $L$-function of a coefficient object. We wanted to mention here an interesting application, even if we will not use it in the article.
	\begin{coro}
		\label{geom-ss-funct-corollary}
	A $\Qlbar$-coefficient object over $X_0$ is geometrically semi-simple if and only if it is a direct sum of $\iota$-pure $\Qlbar$-coefficient objects. In particular, for every morphism $f_0:Y_0\to X_0$ of smooth varieties, if $\calE_0$ is a geometrically semi-simple coefficient object over $X_0$, then $f_0^*\calE_0$ is a geometrically semi-simple coefficient object over $Y_0$.
	\end{coro}
	
	\begin{proof}Combining Corollary~\ref{filt-geo-semi:c}.(ii) and Corollary~\ref{ever-is-mixe:c} we get the first part of the statement. For the second part note that the property of a coefficient object of being a direct sum of $\iota$-pure coefficient objects is manifestly preserved by the inverse image functor $f_0^*$.
	\end{proof}

\begin{rema}
Note that in \cite{Moc}, Mochizuki proves an analogue characterization for semi-simple local systems over smooth quasi-projective varieties over $\C$. Even in that case, one gets as an outcome that the inverse image of a semi-simple local system is semi-simple, [\textit{ibid.}, Theorem 7.1].
\end{rema}

%We end this section with another outcome of Theorem \ref{Deligne-conjecture-general-theorem}. Again, this is obtained thanks to Corollary \ref{rank-1-finite-order:c}.
%	\begin{coro}
%		\label{algebraic-are-rational:coro}
%		For every algebraic $\Qlbar$-coefficient object $\calE_0$, there exists a number field $E\subseteq \Qlbar$, such that $\calE_0$ is $E$-rational.
%	\end{coro}
	
	\subsection{Compatible systems}
	\label{compatible-systems-ss}
	
	We have seen in Theorem~\ref{weak-companions-smooth-theorem}, that from a coefficient object which satisfies certain properties, we can construct many compatible coefficient objects with different fields of scalars. Deligne, in \cite[Conjecture 1.2.10]{Weil2}, predicted that these fields should be the completions at different finite places of a given number field. It is possible to upgrade Theorem~\ref{weak-companions-smooth-theorem} to this stronger form thanks to the work of Chin in \cite{Chin1}. To state the result, we use Serre's notion of \textit{compatible systems}, which we extend to arbitrary coefficient objects as in \cite{Pal}. Since it is not known at the moment whether the $p$-adic companions always exist, we do not ask that compatible systems have to include $p$-adic coefficient objects for every $p$-adic place. On the other hand, we do ask that they include lisse sheaves for every finite place which do not divide $p$.
	
	\begin{defi}[Compatible systems]
		Let $E$ be a number field. An \textit{$E$-compatible system} over $X_0$, denoted by $\underline{\calE_0}$, is the datum of:
		
		\begin{itemize}
\item[--] a set $\Sigma$ of finite places of $E$ which contains $|E|_{\neq p}$,
\item[--] a family $\{\calE_{\lambda, 0}\}_{\lambda\in \Sigma}$ of pairwise $E$-compatible $E_{\lambda}$-coefficient objects, one for each $\lambda\in \Sigma$.
		\end{itemize}
For every $\lambda\in \Sigma$, we say that $\calE_{\lambda, 0}$ is the \textit{$\lambda$-component} of the compatible system. If $E\subseteq E'$ is a finite extension and $\underline{\calE_0}$ is an $E$-compatible system, the compatible system obtained from $\underline{\calE_0}$ by \textit{extending the scalars to} $E'$ is the $E'$-compatible system $\{\calE'_{\lambda',0}\}_{\lambda'\in \Sigma'}$, where $\Sigma'$ is the set of all the places of $E'$ over the places in $\Sigma$ and $\calE'_{\lambda',0}:=\calE_{\lambda, 0}\otimes_{\El} E'_{\lambda'}$ where $\lambda$ denotes the place of $\Sigma$ under $\lambda'$. We say that a compatible system is \textit{trivial, geometrically trivial, pure, irreducible, absolutely irreducible or semi-simple} if each $\lambda$-component has the respective property. 
	\end{defi}
		
	\begin{theo}[after Chin]
		\label{companions-smooth-theorem}
		Let $X_0$ be a smooth variety over $\Fq$ and let $\calE_0$ be an algebraic $\Qlbar$-coefficient object over $X_0$. There exists a number field $E$, a finite place $\nu\in |E|$ and an $E$-compatible system $\underline{\calE_0}$ such that $\calE_0$ is a $\nu$-component of $\underline{\calE_0}$. When $X_0$ is a curve, we can further find such an $E$-compatible system $\underline{\calE_0}$ with $\Sigma=|E|$.
	\end{theo}
	
	\begin{proof}
		It is enough to prove the statement when $\calE_0$ is irreducible. Thanks to Corollary \ref{rank-1-finite-order:c}, $\calE_0$ is isomorphic to $\calF_0^{(a)}$ with $a$ an algebraic number and $\calF_0$ irreducible with finite order determinant. By Theorem \ref{Deligne-conjecture-general-theorem}, the coefficient object $\calF_0$ is $E$-rational for some number field $E$. After extending $E$, we may assume that $a\in E$. If $\ell\neq p$, thanks to Theorem~\ref{weak-companions-smooth-theorem} and \cite[Main Theorem, page 3]{Chin1}, after possibly enlarging $E$ again, the lisse sheaf $\calF_0$ sits in an $E$-compatible system. By twisting all the components by $a$, the same holds true for $\calE_0$. If $\ell=p$, thanks to Theorem~\ref{weak-companions-smooth-theorem}, the coefficient object $\calE_0$ admits an $E$-compatible lisse sheaf $\calV_0$. The result then follows from the previous case.
		
		When $X_0$ is a curve, we obtain the stronger result thanks to the existence of $p$-adic companions provided by Theorem \ref{Langlands-theorem}. After possibly replacing $E$ with a finite extension, we may add to the compatible system previously constructed $\lambda$-components for every place $\lambda$ which divides $p$. Here we do not need a new finiteness result for overconvergent $F$-isocrystals, namely a $p$-adic analogue of Chin's theorem, because the set of places we are adding is finite. 
		
	\end{proof}
	\begin{rema}
		Even if a coefficient object $\calE_0$ is $E$-rational for some number field $E$, it could be still necessary to enlarge $E$ to obtain the $E$-compatible system $\underline{\calE_0}$. For example, let $Q_8$ be the quaternion group and let $X_0$ be a smooth connected variety which admits a Galois cover with Galois group $Q_8$. Let $\mathbb{H}$ be the natural four-dimensional $\Q$-linear representation of $Q_8$ on the algebra of Hamilton's quaternions. 
		
		The representation $\mathbb{H}\otimes_\Q\Ql$ is irreducible over $\Ql$ if and only if $\ell=2$. If we take $\ell\neq 2$, then $\mathbb{H}\otimes_\Q\Ql$ decomposes as a direct sum of two copies of an absolutely irreducible two dimensional $\Ql$-representation $V_\ell$ with traces in $\Q$. The representation $V_\ell$ corresponds to an absolutely irreducible $\Q$-rational $\Ql$-coefficient object which does not admit any $\Q$-compatible $\Q_2$-coefficient object. Indeed, suppose that there exists a semi-simple $\Q_2$-coefficient object $V_2$, that is $\Q$-compatible with $V_\ell$. Then $V_2^{\oplus 2}$ would be $\Q$-compatible with $\mathbb{H}\otimes_\Q\Q_2$. By Proposition~\ref{Chebotarev-proposition}, the coefficient object $V_2^{\oplus 2}$ would be isomorphic to $\mathbb{H}\otimes_\Q\Q_2$. However, this is impossible because $\mathbb{H}\otimes_\Q\Q_2$ is irreducible.
	\end{rema}

	\section{Independence of monodromy groups}
	\label{independence-of-monodromy-section}
	\begin{nota}\label{X:n}Throughout §\ref{independence-of-monodromy-section}, let $(X_0,x)$ be a smooth geometrically connected pointed variety over $\Fq$.
	\end{nota}
	
	\subsection{The group of connected components}
	\label{the-group-of-ss}

	In \cite{Serre} and \cite{LP5} Serre and Larsen--Pink proved some results of $\ell$-independence for the groups of connected components of the monodromy groups of lisse sheaves. In this section, we shall extend their results to algebraic coefficient objects. We will adapt Larsen--Pink's proof. The main issue for $p$-adic coefficient objects is to relate the monodromy groups with the 
	\'etale fundamental group of $X_0$. We have already treated this problem in §\ref{comparison-etale:ss}. By Proposition~\ref{connected-components-etale-proposition}, for every coefficient object $\calE_0$ we have functorial surjective morphisms $\psi_{\calE_0}:\pi^{\et}_1(X_0,x)\to \pi_0(G(\calE_0,x))$ and $\psi_{\calE}:\pi^{\et}_1(X,x)\to \pi_0(G(\calE,x))$ of profinite groups.

	\begin{theo}\label{Larsen-and-Pink-theorem}
	Let $\calE_0$ and $\calF_0$ be two compatible $\El$-coefficient objects over $X_0$.
		\begin{enumerate}
			
			\item[{(i)}]There exists an isomorphism $\varphi_0:\pi_0(G(\calE_0,x))\iso \pi_0(G(\calF_0,x))$ as abstract finite groups such that $\psi_{\calF_{0}} = \varphi_0\circ \psi_{\calE_{0}}$.
			
			\item[{(ii)}] The isomorphism $\varphi_0$ restricts to an isomorphism $\varphi:\pi_0(G(\calE,x))\iso \pi_0(G(\calF,x))$.
		\end{enumerate}
	\end{theo}Following \cite[Proposition~2.2]{LP5}, we need two lemmas to prove Theorem \ref{Larsen-and-Pink-theorem}.
	\begin{cons}
		Let $\calE_0$ be a $\El$-coefficient object of rank $r$. We fix a basis of $\omega_{x,\El}(\calE_0)$ and we write $\rho_{\calE_0}: G(\calE_0,x)\to \GL_{r,\El}$ for the representation associated to $\calE_0$. For every finite-dimensional $\Q$-linear representation $\theta:\GL_{r,\Q}\to \GL(V)$ we denote by $\calE_0(\theta)$ the coefficient object associated to $(\theta\otimes_\Q \El)\circ \rho_{\calE_0}$. Even if $\calE_0(\theta)$ depends on the choice of a basis, its isomorphism class is uniquely determined.
	\end{cons}
	
	\begin{lemm}
		\label{dimension-data:l}
		Let $\calE_0$ and $\calF_0$ be compatible semi-simple objects. For every representation $\theta$ of $\GL_{r,\Q}$, we have $$\dim(H^0(X,\calE(\theta)))=\dim(H^0(X,\calF(\theta)))\textrm{ and } \dim(H^0(X_0,\calE_0(\theta)))=\dim(H^0(X_0,\calF_0(\theta))).$$
	\end{lemm}
	\begin{proof}
		
	The coefficient objects $\calE_0(\theta)$ and $\calF_0(\theta)$ are again compatible and semi-simple. Moreover, by Theorem \ref{Deligne-conjecture-general-theorem}, for every representation $\theta$ the coefficient object $\calE_0(\theta)$ is $\iota$-mixed. Therefore, for every $\theta$, we may apply Corollary~\ref{independence-triviality:c} to $\calE_0(\theta)$ and $\calF_0(\theta)$.
	\end{proof}
	\begin{rema}
		Using the terminology of \cite{LP1}, Lemma \ref{dimension-data:l} proves that $G(\calE,x)$ and $G(\calF,x)$ (resp. $G(\calE_0,x)$ and $G(\calF_0,x)$) have the same \textit{dimension data}.
	\end{rema}
	\begin{lemm}[{\normalfont\cite[Lemma 2.3]{LP5}}]
		\label{connectness-seen-by-dimensions-lemma}Let $\KK$ be a field and $G$ a reductive algebraic subgroup of $\GL_{r,\KK}$. If for every finite-dimensional representation $V$ of $\GL_{r,\KK}$ the dimension of $V^{G^{\circ}}$ is equal to the dimension of $V^{G}$, then the group $G$ is connected.
	\end{lemm}

	\subsubsection{Proof of Theorem~\ref{Larsen-and-Pink-theorem}}
	We explain the proof for the arithmetic monodromy groups. For the geometric ones the proof is the same \textit{mutatis mutandis}.
	
	We notice that taking semi-simplification we do not change the group of connected components of the arithmetic monodromy group. Thus we reduce to the case when $\calF_0$ and $\calG_0$ are semi-simple. We firstly prove a weaker statement.
	\begin{itemize}
		\item[(i')] $G(\calE_0,x)$ is connected if and only if $G(\calF_0,x)$ is connected.
	\end{itemize}
	For every finite étale connected cover $f_{0}: Y_{0} \to X_{0}$, we denote by $a_{Y_{0}}$ and $b_{Y_{0}}$ the functions from the set of isomorphism classes of representations of $\GL_{r,\Q}$ to the natural numbers, defined by $$a_{Y_{0}}(\theta):=\dim(H^0(Y_{0},(f^*_{0}\calE_{0})(\theta))) \textrm{ and }b_{Y_{0}}(\theta):=\dim(H^0(Y_{0},(f^*_{0}\calF_{0})(\theta))).$$  By Lemma~\ref{dimension-data:l}, for every finite étale connected cover $Y_{0}\to X_0$, we have $a_{Y_{0}}=b_{Y_{0}}$. Suppose that $G(\calE_0,x)$ connected. By Proposition~\ref{connected-after-cover-proposition}, for every étale connected cover $f_{0}: (Y_{0},y) \to (X_{0},x)$, the groups $G(f_0^*\calE_0,y)$ and $G(\calE_0,x)$ are isomorphic via the natural morphisms, thus the functions $a_{Y_0}(\theta)$ and $a_{X_0}(\theta)$ are equal. Thanks to Proposition~\ref{connected-components-etale-proposition}, we also know that there exists an étale Galois cover $f_{0}: (Y_0,y) \to (X_{0},x)$ such that $G(f_0^*\calF_0,y)$ is isomorphic to $G(\calF_0,x)^{\circ}$. The functions $b_{Y_0}(\theta)$ and $b_{X_0}(\theta)$ are equal because of the comparison with $a_{Y_0}(\theta)$ and $a_{X_0}(\theta)$. Therefore, by Lemma \ref{connectness-seen-by-dimensions-lemma}, the group $G(\calF_0,x)$ is connected. This concludes the proof of (i').
	
	To prove (i) we show that $\Ker(\psi_{\calE_0})$ and $\Ker(\psi_{\calF_0})$ are the same subgroups of $\pi^{\et}_1(X_0,x)$. By symmetry it is enough to show that $\Ker(\psi_{\calE_0})\subseteq \Ker(\psi_{\calF_0})$. This is equivalent to proving that if $f_0:(Y_{0},y) \to (X_{0},x)$ is the Galois cover associated to $\Ker(\psi_{\calE_0})$, then the natural map $\Ker(\psi_{\calE_0})\to \pi^{\et}_1(X_0,x)/ \Ker(\psi_{\calF_0})$ is the trivial map. In other words, it is enough to show that $G(f_0^*\calF_0,y)$ is connected. As  $G(f_0^*\calE_0,y)$ is connected by construction, this is a consequence of (i').
	
	\qed

	\subsection{Frobenius tori}
	\label{Frobenius-tori:ss}

	We extend here the theory of \textit{Frobenius tori}, previously studied by Serre and Chin in \cite{Serre} and \cite[§5.1]{Chin2}, to algebraic coefficient objects over varieties of arbitrary dimension (Theorem \ref{max-tori:t}). The result for overconvergent $F$-isocrystals is completely new and it is obtained exploiting the existence of étale companions. Once we pass to étale lisse sheaves, we can use Chebotarev's density theorem, which is one of the key ingredients of the proof.

	\begin{cons}[Frobenius tori]\label{Frobenius tori}
		Let $\calE_0$ be an $\El$-coefficient object over $X_0$. For every closed point $i_0:x_0'\hookrightarrow X_0$ we have a functor $\langle \calE_0 \rangle \to \langle i_0^*\calE_0\rangle$ of inverse image. If $x'$ is an $\F$-point over $x_0'$, this functor induces a closed immersion $G(i_0^*\calE_0, x')\hookrightarrow G(\calE_0,x')$. Let $F_{x_0'}$ be the $\Elxp$-point of $G(i_0^*\calE_0, x')$ corresponding to the Frobenius automorphism and let $F_{x_0'}^{\sms}$ be its semi-simple part. The Zariski closure of the group generated by $F_{x_0'}^{\sms}$ is the maximal subgroup of multiplicative type of $G(i_0^*\calE_0, x')$. We call it the \textit{Frobenius group} attached to $x_0'$ and it is denoted by $M(\calE_0,x')$. Its connected component is the \textit{Frobenius torus} attached to $x_0'$, denoted by $T(\calE_0,x')$. If $\calE_0$ is $E$-rational, the torus $T(\calE_0,x')$ descends to a torus $\widetilde{T}(\calE_0,x')$ over $E$, such that $T(\calE_0,x')\simeq \widetilde{T}(\calE_0,x')\otimes_E \Elxp$.
	\end{cons}
	
	To prove our main theorem on Frobenius tori we first need another outcome of Deligne's conjecture. This is a finiteness result for the set of all the possible valuations of the Frobenius eigenvalues at closed points.
	\begin{nota}
		Let us fix a prime $\ell$. For every prime $\ell'$ (even $\ell'=\ell$ or $\ell'=p$), we denote by $I_{\ell'}(\Qlbar)$ the set of field isomorphisms $\Qlbar\iso \Qlpbar$, by $I_\infty(\Qlbar)$ the set of field isomorphisms $\Qlbar\iso \C$ and $$I(\Qlbar):=\left(\bigcup_{\ell'} I_{\ell'}(\calE_0)\right) \cup I_{\infty}(\calE_0).$$ For every $\ell'\neq p$ we endow $\Qlpbar$ with the $\ell'$-adic valuation $v:(\Qlpbar^\times,\times)\to (\R,+)$, normalized such that $v(\ell')=1$. On $\Qpbar$ we consider the $p$-adic valuation $v$, normalized so that $v(q)=1$. Finally, we endow $\C$ with the morphism $v:(\C^{\times},\times)\to (\R,+)$ defined by $a\mapsto \log_q(|a|)$.
	\end{nota}
	\begin{defi} Let $\calE_0$ be a $\Qlbar$-coefficient object. For every closed point $x_0\in |X_0|$, let $A_{x_0}(\calE_0)$ be the set of Frobenius eigenvalues at $x_0$. For $*=\ell',\infty$ we define $$V_*(\calE_0):=\left\{v(\iota(a))/\deg(x_0)\ |\ \ x_0\in |X_0|,\ a\in A_{x_0}(\calE_0),\ \iota\in I_*(\Qlbar) \right\}\subseteq \R.$$
	We also write $$V_{\neq p}(\calE_0):=\left(\bigcup_{\ell' \neq p} V_{\ell'}(\calE_0)\right) \cup V_{\infty}(\calE_0) \textrm{ and } V(\calE_0):=V_{p}(\calE_0)\cup V_{\neq p}(\calE_0).$$

	\end{defi}
	
	\begin{prop}\label{vals-finite:p}If $\calE_0$ be an algebraic $\Qlbar$-coefficient object, the set $V(\calE_0)$ is finite.
	\end{prop}
	\begin{proof} Let $R$ be the ring of integers of a number field. Since $R$ is a Dedekind domain, every element in $R$ belongs to finitely many prime ideals. Moreover, the set of infinite places of $R$ is finite. Therefore, if $a\in\Qlbar$ is algebraic, the set $\{v(\iota(a))\}_{{\iota \in I(\Qlbar)}}$ is finite. Thanks to this, it is harmless to twist our coefficient object by algebraic numbers. By Corollary \ref{rank-1-finite-order:c}, we can then assume that $\calE_0$ is irreducible with finite order determinant. By Theorem \ref{Deligne-conjecture-general-theorem}, the coefficient $\calE_0$ is pure of weight $0$, $E$-rational and $p$-plain. This implies that $V_{\neq p}=\{0\}$ and, by \cite[Lemma-Definition 4.3.2]{Ked}, that the set $V_p(\calE_0)$ is finite. This concludes the proof.
	\end{proof}

	\begin{nota}
	Let $(X_0,x)$ be a smooth connected pointed variety over $\Fq$ and $\calE_0$ an algebraic $\Qlbar$-coefficient object of rank $r$. Write $\GL_{r}$ for the algebraic group $\GL_{r,\Qlbar}$. For every $x'\in X_0(\F)$, we choose an isomorphism between $\omega_{x',\Qlbar}$ and $\omega_{x,\Qlbar}$, which exists by \cite[Theorem 3.2]{DM}, and a basis of $\omega_{x,\Qlbar}(\calE_0)$. This determines in turn an embedding $G(\calE_0,x')\hookrightarrow \GL_r$ for every $x'$. Let $\Gm^r\subseteq \GL_{r}$ be the standard maximal torus and $\chi_1,\dots,\chi_r$ the standard basis of $X^*(\Gm^r)$. The Frobenius torus $T(\calE_0,x)\subseteq G(\calE_0,x)\subseteq \GL_r$ is conjugated to some subtorus $T_{x'}\subseteq\Gm^r$. The torus $T_{x'}$ is uniquely determined up to the action of the permutation group $S_r$ on $\Gm^r$.  We write $C_T(\calE_0)$ for the set of $\GL_r$-conjugacy classes of the Frobenius tori at various $\F$-points of $X_0$. In what follows, for every $x'\in X_0(\F)$, we suppose chosen a representative subtorus $T_{x'}\subseteq\Gm^r$. Besides, we denote by $\alpha_{x',1},\dots,\alpha_{x',r}$ the Frobenius eigenvalues at $x'$. We assume that they are ordered in such a way that the diagonal matrix $(\alpha_{x',1},\dots,\alpha_{x',r})$ is in $T_{x'}(\Qlbar)$.

\end{nota}
	\begin{cons}\label{delta:cons}		 
Let $\Lambda^\vee_{\R}(\calE_0)$ be the set of $\R$-linear subspaces of $X_*(\Gm^r)_\R$ which admit a set of generators in $V(\calE_0)^r\subseteq X_*(\Gm^r)_\R$. For every $Y\subseteq X_*(\Gm^r)_\R$, the natural pairing $(\cdot,\cdot):\ X^*(\Gm^r) \times X_*(\Gm^r) \to \Z$ induces a $\Q$-linear morphism $f_Y: X^*(\Gm^{r})_\Q \to \Hom(Y,\R)$. Write $K_Y\subseteq X^*(\Gm^{r})_\Q$ for the kernel of $f_Y$ and $\Lambda(\calE_0)$ for the set of $\Q$-linear subspaces of $X^*(\Gm^{r})_\Q$ of the form $K_Y$ for some $Y\in \Lambda^\vee_{\R}(\calE_0)$. We have a natural action of $S_r$ on $\Lambda(\calE_0)$ which permutes the coordinates.

For every $\F$-point $x'$, we define $Y_{x'}\subseteq \R^{r}=X_*(\Gm^r)_\R$ as the $\R$-linear subspace generated by the elements ${y}_{x'}^\iota:=\big(y^\iota_{x',1},\dots,y^\iota_{x',r}\big)$, where $\iota$ is an element in $I(\Qlbar)$ and $y^\iota_{x',i}:=v(\iota(\alpha_{x',i}))$. By definition, $Y_{x'}$ is a subspace in $\Lambda^{\vee}_{\R}(\calE_0)$ and the class of $K_{Y_x}$ in $\Lambda(\calE_0)/S_r$ does not depend on the order of the Frobenius eigenvalues $\alpha_{x',1},\dots,\alpha_{x',r}$.
	\end{cons}
	\begin{prop}\label{injective-map:p}
		For every coefficient object $\calE_0$ there exists an injective map of sets $\delta:C_T(\calE_0)\to \Lambda(\calE_0)/S_r$ which sends the conjugacy class of a Frobenius torus $T_{x'}$ to the class of $K_{Y_{x'}}\subseteq X^*(\Gm^{r})_\Q$ (cf. Construction \ref{delta:cons}).
	\end{prop}
	\begin{proof}
		Let $K_{x'}\subseteq X^*(\Gm^r)_\Q$ be the $\Q$-linear subspace generated by all the $\chi\otimes 1\in  X^*(\Gm^r)_\Q$ such that the restriction $\chi|_{T_{x'}}$ of finite order. First, let us prove that $K_{x'}=K_{Y_{x'}}$.

		Note that it is enough to prove the equality after intersecting both $\Q$-vector spaces with the lattice $X^*(\Gm^r)\subseteq X^*(\Gm^r)_\Q$. Let us first prove that $K_{x'}\subseteq K_{Y_{x'}}$, namely that $f_{Y_{x'}}(K_{x'})=0$. For every character $\chi=\chi_1^{\otimes a_1}\otimes\cdots\otimes \chi_r^{\otimes a_r}\in X^*(\Gm^{r})\cap K_{x'}$, we have that $\beta_{x'}:=\alpha_{{x'},1}^{a_1}\cdot \ldots \cdot \alpha_{{x'},r}^{a_r}$ is a root of unity because $(\alpha_{{x'},1},\dots,\alpha_{{x'},r})\in T_{x'}(\Qlbar)$. Therefore, for every $\iota\in I(\Qlbar)$, $$f_{Y_{x'}}(\chi)(y_{x'}^\iota)=a_1y_{{x'},1}^\iota+\dots+a_ry_{{x'},r}^\iota=v(\iota(\beta_{x'}))=0.$$ This implies that $f_{Y_{x'}}(\chi)=0$. On the other hand, for every character $\chi=\chi_1^{\otimes a_1}\otimes \ldots \otimes \chi_r^{\otimes a_r}\in X^*(\Gm^{r})\cap K_{Y_{x'}}$ we want to show that the restriction of $\chi$ to $T_{x'}$ is finite. Since the subgroup generated by the point $(\alpha_{{x'},1},\dots,\alpha_{{x'},r})\in T_{x'}(\Qlbar)$ is Zariski dense in $T_{x'}$, it is enough to show that $\beta_{x'}:=\alpha_{{x'},1}^{a_1}\dots\alpha_{{x'},r}^{a_r}$ is a root of unity. For every $\iota\in I(\Qlbar)$, we have that $v(\iota(\beta_{x'}))=a_1y_{{x'},1}^\iota+\dots+a_ry_{{x'},r}^\iota=f_{Y_{x'}}(\chi)(y_{x'}^\iota)=0$. By Kronecker's theorem, it follows that $\beta_{x'}$ is a root of unity.
		
		Let us show now how the previous claim can be used to prove the final statement. Since $K_{x'}=K_{Y_{x'}}$, it follows that $X^*(T_{x'})_\Q=X^*(\Gm^r)_\Q/K_{Y_{x'}}$. Therefore, if $x''$ is another $\F$-point of $X_0$, we have that $T_{x'}$ and $T_{x''}$ are conjugated if and only if $K_{Y_{x'}}$ and $K_{Y_{x''}}$ are the same up to permutation of the coordinates. This proves that $\delta$ is a well defined injective map.
	\end{proof}
	
	\begin{coro}\label{finiteness-conj-Frob:c}
		Let $\calE_0$ be an algebraic $\Qlbar$-coefficient object. The set $C_T(\calE_0)$ is finite.
	\end{coro}
	\begin{proof}
		By Proposition \ref{vals-finite:p}, the set $V(\calE_0)$ is finite. Therefore, by construction, the sets $\Lambda^\vee_{\R}(\calE_0)$ and $\Lambda(\calE_0)$ are finite as well. Thanks to Proposition \ref{injective-map:p}, this shows that $C_T(\calE_0)$ is finite.
	\end{proof}

	From here, thanks to Chebotarev's density theorem, we could directly prove Theorem \ref{max-tori:t} for étale lisse sheaves as in the proof of \cite[Théorème at page 12]{Serre}. The result for non-étale lisse sheaves and overconvergent $F$-isocrystals follows from further two crucial facts.

	\begin{prop}[after Larsen--Pink]\label{max-tori-LP:p}
		Let $\calE_0$ and $\calF_0$ be two compatible coefficient objects over $X_0$. The reductive ranks of $G(\calE_0,x)$ and $G(\calF_0,x)$ are equal. 
	\end{prop}
	\begin{proof}We may assume that $\calE_0$ and $\calF_0$ are semi-simple, because semi-simplification keeps the reductive ranks of the monodromy groups unchanged. Moreover, thanks to Proposition~\ref{connected-after-cover-proposition}, we may assume that the monodromy groups of $\calE_0$ and $\calF_0$ are connected. We choose embeddings of the fields of scalars of $\calE_0$ and $\calF_0$ to $\C$ in such a way that the characteristic polynomials at closed points of the two coefficient objects are the same as polynomials in $\C[t]$. We write $G(\calE_0,x)_{\C}$ and $G(\calF_0,x)_{\C}$ for the base change of the monodromy groups of $\calE_0$ and $\calF_0$ to $\C$ and $\rho_{\calE_0,\C}$ and $\rho_{\calF_0,\C}$ for their tautological $\C$-linear representations. Thanks to Lemma \ref{dimension-data:l}, the pairs $(G(\calE_0,x)_{\C},\rho_{\calE_0,\C})$ and $(G(\calF_0,x)_{\C},\rho_{\calF_0,\C})$ satisfy the hypothesis of \cite[Proposition 1]{LP1}. Therefore, $G(\calE_0,x)_{\C}$ and $G(\calF_0,x)_{\C}$ have isomorphic maximal tori. This gives the desired result.
	\end{proof}

\begin{lemm}\label{almost-etale:l}
Let $\underline{\calE_0}$ be an $E$-compatible system. For all but finitely many $\lambda\in |E|_{\neq p}$, the $\lambda$-component of $\underline{\calE_0}$ is an étale lisse sheaf.
\end{lemm}
\begin{proof}
	
	Let $\calE_0$ be a component of $\underline{\calE_0}$. By Corollary \ref{rank-1-finite-order:c} we know that $\calE_0^{\sms}\simeq \bigoplus_{i=1}^n \calF_{i,0}^{(a_i)}$ where for every $i$, the coefficient object $\calF_{i,0}$ is absolutely irreducible with finite order determinant. By Theorem \ref{Deligne-conjecture-general-theorem}, the coefficient objects $\calF_{i,0}$ are $p$-plain. In addition, since $\calE_0$ is algebraic, we know that each $a_i$ is an algebraic number. Arguing as in Proposition \ref{vals-finite:p}, this implies that for all but finitely many primes $\ell\neq p$, the algebraic numbers $a_i$ are $\ell$-adic units. Therefore, for all but finitely many $\ell\neq p$, the Frobenius eigenvalues of $\calE_0$ are $\ell$-adic units. By compatibility we deduce that for all but finitely $\lambda\in |E|_{\neq p}$, the Frobenius eigenvalues of the $\lambda$-component of $\underline{\calE_0}$ are $\lambda$-adic units. By Proposition \ref{p-plain-etale:p}, this implies that all these $\lambda$-components are étale lisse sheaves. This yields the desired result.
\end{proof}

	\begin{theo}\label{max-tori:t}
		Let $X_0$ be a smooth connected variety over $\Fq$ and $\calE_0$ an algebraic coefficient object. There exists a Zariski-dense subset $\Delta\subseteq X(\F)$ such that for every $\F$-point $x'\in \Delta$ and every object $\calF_0\in \langle \calE_0 \rangle$, the torus $T(\calF_0,x')$ is a maximal torus of $G(\calF_0,x')$. Moreover, if $\calG_0$ is a coefficient object compatible with $\calE_0$, the subset $\Delta$ satisfies the same property for the objects in $\langle \calG_0 \rangle$.
	\end{theo}
	\begin{proof}
		Let $x'$ be a geometric point and $i_0:x'_0\hookrightarrow X_0$ the embedding of the underlying closed point. For every object $\calF_0\in\langle \calE_0 \rangle$, we have a commutative square of functors
		\begin{center}
			\begin{tikzcd}
				\langle\calF_0\rangle \arrow[hook,r] \arrow[d,"i_0^*"] &  \langle\calE_0 \rangle \arrow[d,"i_0^*"]\\
				\langle i_0^*\calF_0\rangle \arrow[hook,r] & \langle i_0^*\calE_0 \rangle.\
			\end{tikzcd}
		\end{center}
		It induces a square on monodromy groups 
		\begin{center}
			\begin{tikzcd}
				G(\calF_0,x') & \arrow[two heads, l]    G(\calE_0,x')\\
				M(\calF_0,x')\arrow[u,hook,"i_{0*}"] &     \arrow[two heads,l]M(\calE_0,x')\arrow[u,hook,"i_{0*}"] .\
			\end{tikzcd}
		\end{center}
		If $T(\calE_0,x')$ is a maximal torus in $G(\calE_0,x')$, then the same is true for $T(\calF_0,x')$ in $G(\calF_0,x')$, \cite[Proposition 11.14.(1)]{Bor91}. This shows that it is enough to prove the result when $\calF_0=\calE_0$. Moreover, we may assume that $\calE_0$ is semi-simple, because semi-simplification does not change the reductive rank of the monodromy group. 
		
		We notice that by Proposition \ref{max-tori-LP:p}, if $\calG_0$ is a coefficient object compatible with $\calE_0$, the torus $T(\calE_0,x')$ is maximal in $G(\calE_0,x')$ if and only if $T(\calG_0,x')$ is maximal in $G(\calG_0,x')$. Therefore, it is enough to prove the result for some coefficient object compatible with $\calE_0$. By Theorem \ref{companions-smooth-theorem}, $\calE_0$ sits in a semi-simple compatible system $\underline{\calE_0}$. By Lemma \ref{almost-etale:l}, there exists a component of $\underline{\calE_0}$ which is an étale lisse sheaf. Let us denote it by $\calV_0$. After replacing $X_0$ with a connected finite étale cover we may assume by Proposition \ref{connected-after-cover-proposition} that $G(\calV_0,x')$ is connected for any choice of $x'$. We choose an $\F$-point $x'$ of $X$. By Corollary \ref{finiteness-conj-Frob:c}, the set of conjugacy classes of Frobenius tori $T(\calV_0,x')$ in $\GL(\omega_{x',\Qlbar}(\calV_0))$, where $x'$ varies among the $\F$-points of $X_0$, is finite. 
		Arguing as in \cite[theorem at page 12]{Serre}, by Chebotarev's density theorem for the étale fundamental group of $X_0$, there exists a Zariski-dense subset $\Delta\subseteq X(\F)$ such that for every $\F$-point $x\in \Delta$, the torus $T(\calV_0,x)$ is maximal inside $G(\calV_0,x)$ (see also \cite[Theorem 5.7]{Chin2} for more details). This concludes the proof.
	\end{proof}

Theorem \ref{max-tori:t} is a crucial step in the proof of Theorem \ref{neutral-component-theorem}. Nonetheless, it has also its own interest. In \cite{AD18}, it is used as a starting point to study the reductive rank of the monodromy groups of \textit{convergent $F$-isocrystals} which admit an overconvergent extension. Here another consequence.
	\begin{coro}\label{semi-simple-at-points:c}
	Let $\calE_0$ be a semi-simple $\Qlbar$-coefficient object. The set of closed points $x_0'$ of $X_0$ where the Frobenius $F_{x'_0}$ is regular semi-simple is Zariski-dense in $X_0$.
\end{coro}
\begin{proof}
	As usual, let us start out with some preliminary reductions. We first observe that by \cite[Proposition 12.4.(2)]{Bor91}, if the statement is true for $\calE_0$, then the same is true for every $\calF_0\in \langle \calE_0 \rangle$. In addition, if we know the result for $\calE_0$ and $\calF_0$ at the same time, we can deduce it for $\calE_0\oplus \calF_0$ since $G(\calE_0\oplus \calF_0)\subseteq G(\calE_0)\times G(\calF_0)$ and the property of being a regular semi-simple element is preserved by taking products. Finally, note that for a coefficient object of the form $\Qlbar^{(a)}$ with $a\in \Qlbar$ the result is trivial. Combining these three facts and Corollary \ref{rank-1-finite-order:c} we deduce that it is harmless to assume that $\calE_0$ is an irreducible coefficient object with finite order determinant. By Theorem \ref{Deligne-conjecture-general-theorem}, the coefficient object $\calE_0$ is then algebraic. Thanks to Theorem \ref{max-tori:t}, the set of $\F$-points $x'$ such that the torus $T(\calE_0,x')$ is a maximal torus of $G(\calE_0,x')$ is Zariski-dense in $X_0$. We claim that if $x'$ is one of these $\F$-points, the Frobenius $F_{x'_0}$ is regular semi-simple. Let $Z(F_{x'_0}^{\mathrm{ss}})$ be the centraliser of $F_{x'_0}^{\mathrm{ss}}$ in $G(\calE_0,x')$. By construction, $Z(F_{x'_0}^{\mathrm{ss}})$ coincides with the centraliser of $M(\calE_0,x')$ in $G(\calE_0,x')$. Since $G(\calE_0,x')$ is reductive and $M(\calE_0,x')$ contains a maximal torus, we deduce that $Z(F_{x'_0}^{\mathrm{ss}})^\circ=T(\calE_0,x')$. If we write $F_{x'_0}^{\mathrm{u}}$ for the unipotent part of the multiplicative Jordan form of $F_{x'_0}$, we know that it lies in $Z(F_{x'_0}^{\mathrm{ss}})^\circ(\Qlbar)$. This implies that $F_{x'_0}^{\mathrm{u}}=1$, thus $F_{x_0'}$ is regular semi-simple.
\end{proof}
	
			\begin{rema}
			
			In the proof of Theorem \ref{max-tori:t}, we need Deligne's conjecture in order to prove the following three properties of the coefficient object $\calE_0$.
			\begin{itemize}
				\item[{(i)}] $\calE_0$ is $\iota$-mixed.
				\item[{(ii)}] $V(\calE_0)$ is a finite set.
				\item[\normalfont(iii)] $\calE_0$ admits a compatible étale lisse sheaf.
			\end{itemize}
			For many coefficient objects “coming from geometry”, it is possible to prove these properties directly, without using Theorem \ref{Langlands-theorem}.
		\end{rema}
	\subsection{The neutral component}
	\label{the-neutral-component-ss}
	We start with a first result on the independence of the neutral components of the monodromy groups of coefficient objects. As in Theorem \ref{Larsen-and-Pink-theorem}, the independence result we need here is Corollary~\ref{independence-triviality:c}. 
	\begin{prop}
		\label{weak-neutral-component-p}
		Let $X_0$ be a smooth geometrically connected variety over $\Fq$. Let $\calE_0$ and $\calF_0$ be two compatible coefficient objects over $X_0$.
		\begin{enumerate}
			\item[{(i)}] If $\calE_0$ and $\calF_0$ are semi-simple, $\calE_0$ is finite if and only if $\calF_{0}$ is finite.
			\item[{(ii)}]If $\calE_0$ and $\calF_0$ are geometrically semi-simple, $\calE$ is finite if and only if $\calF$ is finite.

		\end{enumerate}
	\end{prop}
	
	\begin{proof}
		Thanks to Proposition~\ref{connected-after-cover-proposition}, we may assume that the arithmetic and the geometric monodromy groups of $\calE_0$ and $\calF_0$ are connected. By Theorem \ref{Deligne-conjecture-general-theorem}, we know that $\calE_0$ and $\calF_0$ are $\iota$-mixed. Thus, thanks to Corollary~\ref{independence-triviality:c}, the coefficient object $\calE_0$ is trivial (resp. geometrically trivial) if and only if the same is true for $\calF_0$. 
	\end{proof}

	The next result we want to prove is a generalization of \cite[Theorem~1.4]{Chin2}. 
	
	\begin{theo}
		\label{neutral-component-theorem}
Let $(X_0,x)$ be a smooth connected pointed variety over $\Fq$, $E$ a number field and $\underline{\calE_{0}}$ a semi-simple $E$-compatible system over $X_0$. For every $\lambda\in \Sigma$, let $\rho_{\lambda,0}$ be the associated representation on $\omega_{x,E_{\lambda}}(\calE_{\lambda, 0})$. After possibly replacing $E$ by a finite extension, there exists a connected split reductive group $G_0$ over $E$ such that, for every $\lambda\in \Sigma$, the extension of scalars $G_0\otimes_E \El$ is isomorphic to $G(\calE_{\lambda, 0},x)^\circ$. Moreover, there exists a faithful $E$-linear representation $\rho_0$ of $G_0$ and isomorphisms $\varphi_{\lambda,0}: G_{0}\otimes_E E_{\lambda} \xrightarrow{\sim} G(\calE_{\lambda, 0},x)^{\circ}$ for every $\lambda\in \Sigma$ such that $\rho_0\otimes_E E_{\lambda}$ is isomorphic to $\rho_{\lambda,0}\circ \varphi_{\lambda,0}$.
	\end{theo}

	Following Chin, we use a reconstruction theorem of a reductive group from the Grothendieck semiring of its category of finite-dimensional representations.
	
		\begin{nota}
		If $\bfC$ is a Tannakian category, we denote by $K^+(\bfC)$ its \textit{Grothendieck semiring}, namely the semiring of isomorphism classes of semi-simple objects of $\bfC$ with sum and product induced by $\oplus$ and $\otimes$. If $\calE_0$ is a coefficient object and $\bfC=\langle \calE_0 \rangle$, we denote $K^+(\bfC)$ by $K^+(\calE_0)$. Finally, when $\bfC=\Rep(G)$ with $G$ an algebraic group, we will write $K^+(G)$.
	\end{nota}

	\begin{theo}[{\normalfont\cite[Theorem~1.4]{Chin3}}]
		
		\label{Chin-theorem}
		Let $G$ and $G'$ be two connected split reductive groups, defined over a field $\KK$ of characteristic $0$. Let $T$ and $T'$ be maximal tori of $G$ and $G'$ respectively. For every pair of isomorphisms $\varphi_{T'}: T'\iso T$ and $f: K^+(G)\xrightarrow{\sim} K^+(G')$ making the following diagram commuting
		\begin{center}
			\begin{tikzcd}
				K^+(G) \arrow{r}{f} \arrow[d] & K^+(G')\arrow[d]\\
				K^+(T)  \arrow[r,"\varphi_{T'}^*"] & K^+(T'),\
			\end{tikzcd}
		\end{center}
		there exists an isomorphism $\varphi: G'\xrightarrow{\sim} G$ of algebraic groups such that the induced homomorphism $\varphi^*$ on the Grothendieck semirings is equal to $f$ and the restriction of $\varphi$ to $T'$ is equal to $\varphi_{T'}$.
	\end{theo}
	
		\begin{rema}
		\label{tori-compatible-remark}
		
		The maximal tori that we will use to apply Theorem \ref{Chin-theorem} will be the Frobenius tori provided by Theorem \ref{max-tori:t}. Suppose that $\calE_{0}$ is a coefficient object and for some $\F$-point $x'$, the group $M(\calE_0,x')$ is connected and $\widetilde{T}(\calE_0,x')$ is a split torus over $E$. Then, the group of characters of $\widetilde{T}(\calE_0,x')$ is canonically isomorphic to the subgroup of $E^\times$ generated by the eigenvalues of $F_{x_0'}$. The isomorphism is given by the evaluation of a character at the point $F_{x_0'}^{\sms}$. In particular, if $\calE_0$ sits in an $E$-compatible system $\underline{\calE_0}$ and $\widetilde{T}(\calE_{\lambda, 0},x')$ is split over $E$ for one $\lambda\in \Sigma$ (or equivalently every $\lambda\in \Sigma$), the semirings $K^+(\widetilde{T}(\calE_{\lambda, 0},x'))$ are all canonically isomorphic when $\lambda$ varies in $\Sigma$. Moreover, note that for every $\lambda\in \Sigma$, the semiring $K^+(\widetilde{T}(\calE_{\lambda, 0},x'))$ is canonically isomorphic to $K^+(T(\calE_{\lambda, 0},x'))$.
	\end{rema}

	The known cases of the companions conjecture provide isomorphisms of the Grothendieck semiring of compatible objects. A bit surprisingly, we have these isomorphisms even if we do not dispose at the moment of a general way to construct compatible overconvergent $F$-isocrystals in dimension greater than $1$.

	\begin{prop}
		\label{consequence-companions-proposition}
		Let $\calE_0$ and $\calF_0$ be two $E$-compatible coefficient objects such that all the irreducible objects in $\langle\calE_0\rangle$ and $\langle\calF_0 \rangle$ are absolutely irreducible. There exists an isomorphism of semirings $K^+(\calE_0)\iso K^+(\calF_0)$ which sends the isomorphism class of a semi-simple coefficient object in $\langle\calE_0\rangle$ to the isomorphism class of an $E$-compatible semi-simple coefficient object in $\langle\calF_0\rangle$.
	\end{prop}

	\begin{proof}
		
		Let us first construct the isomorphism when $\calF_0$ is a lisse sheaf. Thanks to Theorem~\ref{companions-smooth-theorem}, there exists a morphism of semirings $f:K^+(\calE_0)\to K^+(\calF_0)$ which sends semi-simple coefficient objects to $E$-compatible semi-simple coefficient objects. By Theorem~\ref{Chebotarev-proposition} this morphism is an injective morphism. Let us show that it is surjective as well. By Corollary~\ref{independence-triviality:c} and the hypothesis, $f$ sends irreducible objects to irreducible objects. Therefore, if $[\calH_0]\in K^+(\calE_0)$ and $\sum_{i=0}^n m_i[\calH_{0}^i]$ is the isotypic decomposition of $[\calH_0]$, then $\sum_{i=0}^n m_i f([\calH_{0}^i])$ is the isotypic decomposition of $f([\calH_0])$. In particular, every summand of a class in the image of $f$ is again in the image of $f$. On the other hand, we know that for every $n,m\in \mathbb{N}$, the classes $\left[\calF_0^{\otimes n} \otimes (\calF_0^{\vee})^{\otimes m}\right]$ are clearly in the image of $f$. Combining these two facts we get the surjectivity. 
		
		If $\calF_0$ is an overconvergent $F$-isocrystals instead, thanks to Theorem~\ref{companions-smooth-theorem}, there exists a compatible lisse sheaf $\calG_0$. The isomorphism $K^+(\calE_0)\iso K^+(\calF_0)$, is then obtained via the composition $$K^+(\calE_0)\iso K^+(\calG_0)\iso K^+(\calF_0).$$
	\end{proof}
	\begin{rema}
		The assumption that the irreducible objects in $\langle\calE_0\rangle$ and $\langle\calF_0 \rangle$ are absolutely irreducible is verified, for example, when $G(\calE_0,x)$ and $G(\calF_0,x)$ are split reductive groups. In particular, it is always possible to obtain this condition after a finite extension of the fields of scalars of the coefficient objects.
	\end{rema}
	
	\subsubsection{Proof of Theorem~\ref{neutral-component-theorem}}
	\label{neutral-for-curves:pf}
	Thanks to Theorem~\ref{Larsen-and-Pink-theorem}, there exists a Galois cover of $X_0$ such that all the arithmetic monodromy groups of the compatible system are connected. By Proposition~\ref{connected-after-cover-proposition}, the neutral components of the monodromy groups remain unchanged when we base change the compatible system to such a cover. By Remark \ref{base-point-r}, after possibly passing to a finite extension of $E$, we may change the $\F$-point $x$ without changing the isomorphism class of the monodromy groups. We may also assume that $\El^{(x_0)}=\El$ for every $\lambda\in \Sigma$. Thanks to Theorem~\ref{max-tori:t}, we may choose $x$ so that $T(\calE_{\lambda, 0},x)$ is a maximal torus in $G(\calE_{\lambda, 0},x)$ for every $\lambda\in \Sigma$. Besides, up to replacing $E$ again with a finite extension, we may assume that $\widetilde{T}(\calE_{\lambda, 0},x)$ is split over $E$ for every $\lambda\in \Sigma$. We fix a finite place $\mu\in \Sigma$. By \cite[Corollaire 1.3]{Dem70}, there exists a connected split reductive group $G_0$ over $E$ such that $G_0\otimes_E E_{\mu}\simeq G(\calE_{\mu,0},x)$. Let $T_0$ be a split maximal torus of $G_0$. We choose $\rho_0: G_0\hookrightarrow \GL_{r,E}$ and $\varphi_{\mu,0}: G_0\otimes_E E_{\mu}\iso G(\calE_{\mu,0},x)$ such that $\varphi_{\mu,0}(T_0\otimes_E E_{\mu})=T(\calE_{\mu,0},x)$ and $\rho_0\otimes_E E_{\mu} \simeq \rho_{\mu,0}\circ \varphi_{\mu,0}$. The isomorphism $\varphi_{\mu,0}$ induces an isomorphism $$\varphi_{\mu,0}^*: K^+(\calE_{\mu,0})\iso K^+(G_0\otimes_E E_{\mu})$$ which sends $[\calE_{\mu,0}]$ to $[\rho_0\otimes_E E_{\mu}]$. As $\widetilde{T}(\calE_{\lambda, 0},x)$ is split over $E$ for every $\lambda\in \Sigma$, the reductive groups $G(\calE_{\lambda, 0},x)$ are all split. By Proposition~\ref{consequence-companions-proposition}, for every $\lambda\in \Sigma$, there exists a unique isomorphism $g_{\lambda,\mu}: K^{+}(\calE_{\lambda, 0})\simeq K^+(\calE_{\mu,0})$ preserving the characteristic polynomials at closed points, hence sending $[\calE_{\lambda, 0}]$ to $[\calE_{\mu,0}]$. Since $G_0$ is split reductive and connected, there exists a canonical isomorphism $h_{\mu,\lambda}:K^+(G_0\otimes_E E_{\mu})\iso K^+(G_0\otimes_E E_{\lambda})$. We take $$f_{\lambda,0}:=h_{\mu,\lambda}\circ \varphi_{\mu,0}^*\circ g_{\lambda,\mu}: K^+(\calE_{\lambda, 0})\iso K^+(G_0\otimes_E E_{\lambda}).$$ By construction, it is compatible with the isomorphism $K^+(T(\calE_{\lambda, 0},x))\simeq K^+(T_0\otimes_E \El)$ induced by $\varphi_{\mu,0}^*$ and the identifications of Remark \ref{tori-compatible-remark}. Thanks to Theorem~\ref{Chin-theorem}, the isomorphism $f_{\lambda,0}$ induces an isomorphism $\varphi_{\lambda,0}: G_{0}\otimes_E E_{\lambda} \iso G(\calE_{\lambda, 0},x)$ such that $f_{\lambda,0}=\varphi_{\lambda,0}^*$. Since $f_{\lambda,0}([\calE_{\lambda, 0}])=[\rho_0\otimes_E \El]$, the representations $\rho_{0} \otimes_E \El$ and $\rho_{\lambda,0}\circ \varphi_{\lambda,0}$ are isomorphic.	
	\qed

	\begin{rema}\label{geometric-neutral-remark}As a consequence of Theorem~\ref{neutral-component-theorem}, we obtain the analogue $\lambda$-independence result for the geometric monodromy groups of the compatible system. Indeed, by Corollary \ref{derived-sub:c}, if $\calE_0$ is a geometrically semi-simple coefficient object, $G(\calE,x)^\circ$ is the derived subgroup of $G(\calE_0,x)^\circ$.
	\end{rema}
	
	\begin{rema}\label{work-drinfeld-remark}
		If we weaken the statement of Theorem~\ref{neutral-component-theorem}, asking that all the isomorphisms between $G_0$ and the monodromy groups are defined over $\Qlbar$, rather than $E_{\lambda}$, one can prove it differently. One can use \cite[Theorem~1.2]{KLV}, a stronger version of Theorem~\ref{Chin-theorem}, in combination with Proposition~\ref{consequence-companions-proposition}. This proof does not use Frobenius tori. 
		
		The author became aware of the theorem of Kazhdan--Larsen--Varshavsky reading \cite{Dri2}. In his paper, Drinfeld uses this result as a starting point to prove the independence of the entire monodromy groups over $\Qlbar$ (not only the neutral components). 
	\end{rema}

	\subsection{Restriction to curves}
	
	\label{Lefschetz-theorem:ss}
	
	In this section, we prove a $\lambda$-independence result for Theorem~\ref{dom-curve-theorem}. We give a proof which exploits the full strength of the Tannakian lemma \cite[Lemma 1.6]{AE}. A similar argument is used in [\textit{ibid.}, Corollary 3.7]. We also need a lemma which relates the arithmetic and the geometric situation.
	
	\begin{lemm}
		\label{isogeom-isoarit-lemma}
		Let $(Y_0,y)\to (X_0,x)$ be a morphism of geometrically connected pointed varieties over $\Fq$ and $\calE_0$ a coefficient object over $X_0$.
		\begin{itemize}
			\item[\normalfont(i)]If the natural morphism $f_*:G(f^*\calE,y)\to G(\calE,x)$ is an isomorphism, the same is true for $f_{0*}:G(f_0^*\calE_0,y)\to G(\calE_0,x)$. 
			\item[\normalfont(ii)] If $\calE_0$ is geometrically semi-simple and $f_{0*}:G(f_0^*\calE_0,y)^\circ\to G(\calE_0,x)^\circ$ is an isomorphism, even $f_*:G(f^*\calE,y)^\circ\to G(\calE,x)^\circ$ is an isomorphism.

		\end{itemize}  
	\end{lemm}
	
	\begin{proof}
		We want to use the functorial diagram of Proposition \ref{fun-exact-seq:p}.(iii) to show that the morphism $f_{0*}$ in (i) is surjective. Since $Y_0$ and $X_0$ are geometrically connected, $f_{0*}: \pil(Y_0,y)^{\cst}\to \pil(X_0,x)^{\cst}$ is an isomorphism, therefore $f_{0*}: G(f^*\calE_0,y)^{\cst}\to G(\calE_0,x)^{\cst}$ is surjective. On the other hand, at the level of geometric monodromy groups, $f_{*}:G(f^*\calE,y)\to G(\calE,x)$ is surjective by assumption. The surjectivity of $f_{0*}:G(f_0^*\calE_0,y)\to G(\calE_0,x)$ is then a consequence of the other two. For (ii) we note that by Corollary \ref{derived-sub:c}, the algebraic groups $G(f^*\calE,y)^\circ$ and $G(\calE,x)^\circ$ are the derived subgroups of $G(f_0^*\calE_0,y)^\circ$ and $G(\calE_0,x)^\circ$ respectively. Thus we get the result.
	\end{proof}
	
	\begin{theo}
		\label{isoindipendent:t}Let $f_0:(Y_0,y)\to (X_0,x)$ be a morphism of smooth geometrically connected pointed varieties. Let $\calE_0$ and $\calF_0$ be compatible  geometrically semi-simple coefficient objects over $X_0$. Let $\varphi_0: G(f_0^*\calE_0,y)\to G(\calE_0,x)$ and $\psi_0:G(f_0^*\calF_0,y)\to G(\calF_0,x)$ be the morphisms induced by $f_0^*$ and let $\varphi$ and $\psi$ be their restriction to the geometric monodromy groups.
		\begin{itemize}
			\item[{(i)}]If $\varphi$ is an isomorphism, the same is true for $\psi$.
			\item[{(ii)}] If $\varphi_0$ is an isomorphism, the same is true for $\psi_0$.
			
		\end{itemize}
		
	\end{theo}
	
	\begin{proof}
		By Lemma \ref{isogeom-isoarit-lemma}, part (i) implies part (ii). Note that $\varphi$ and $\psi$ are always injective, thus to prove part (i) it is enough to prove that if $\varphi$ is surjective, the same is true for $\psi$. Suppose that $\varphi$ is surjective, we want to apply \cite[Lemma 1.6] {AE} to prove that $\psi$ is surjective as well. Since, by Theorem~\ref{rank-1-finite-order:t}, the functor $f^*:\langle \calF \rangle \to \langle f^* \calF \rangle$ satisfies the hypothesis $(\star)$ of the lemma, it remains to show that it is also fully faithful.
		
		A functor of Tannakian categories commuting with fibre functors is always faithful. Therefore, it is enough to prove that $f^*$ preserves the dimensions of the Hom-sets, or equivalently that for every $\calG\in \langle \calF \rangle$ we have 
		\begin{equation}
		\label{equation}
		h^0(\calG)= h^0(f^*\calG),
		\end{equation} 
		\noindent
		where we denote by $h^0$ the dimension of the space of global sections of geometric coefficient objects. 
		
		We proceed by steps. First we prove that for every pair of coefficient objects $\calG', \calG''\in \langle\calF\rangle$, they both satisfy the equality (\ref{equation}) if and only if the same is true for $\calG'\oplus \calG''$. By the additivity of $h^0$, it is clear that if two geometric coefficient objects satisfy the equality individually, then the same is true for their direct sum. Conversely, if $h^0(\calG'\oplus \calG'')=h^0(f^*(\calG'\oplus\calG'')),$ then $$h^0(\calG')-h^0(f^*\calG')+h^0(\calG'')-h^0(f^*\calG'')=0.$$
		Since $f^*$ is faithful, then $h^0(\calG')-h^0(f^*\calG')\leq 0$ and  $h^0(\calG'')-h^0(f^*\calG'')\leq 0$, therefore $h^0(\calG')=h^0(f^*\calG')$ and $h^0(\calG'')=h^0(f^*\calG'')$, as we wanted. In particular, since $\langle \calF \rangle$ is a semi-simple category, we have proven that it is enough to show (\ref{equation}) for the objects of the form $\calF^{\otimes m}\otimes (\calF^{\vee})^{\otimes n}$ with $m,n\in \N$.
		
		We fix $m,n\in \N$. By the hypothesis, the $\otimes$-functor $f^*: \langle \calE \rangle \to \langle f^*\calE \rangle$ is fully faithful, therefore the equality (\ref{equation}) holds for $\calE^{\otimes m}\otimes (\calE^{\vee})^{\otimes n}$. By Corollary \ref{ever-is-mixe:c}, we know that every coefficient object is $\iota$-mixed. Therefore, by Corollary \ref{independence-triviality:c}, we have that $h^0(\calE^{\otimes m}\otimes (\calE^{\vee})^{\otimes n})=h^0(\calF^{\otimes m}\otimes (\calF^{\vee})^{\otimes n})\textrm{ and
		}h^0(f^*(\calE^{\otimes m}\otimes (\calE^{\vee})^{\otimes n}))=h^0(f^*(\calF^{\otimes m}\otimes (\calF^{\vee})^{\otimes n})).$
		Hence we get $h^0(\calF^{\otimes m}\otimes (\calF^{\vee})^{\otimes n})= h^0(f^*(\calF^{\otimes m}\otimes (\calF^{\vee})^{\otimes n})).$ This concludes the proof.
	\end{proof}
	\begin{rema}
		In the proof presented here, we use the known cases of Deligne's conjecture only in the end, in order to prove that $\calE_0$ and $\calF_0$ are $\iota$-mixed. If one already knows that the coefficient objects are $\iota$-mixed, then it is not necessary to invoke Theorem \ref{Deligne-conjecture-general-theorem}. In this case, Theorem \ref{isoindipendent:t} is proven avoiding the results in §\ref{Deligne-conjectures-ss}.
	Alternatively, one could get Theorem \ref{isoindipendent:t} from Theorem \ref{neutral-component-theorem}. This other proof makes substantial use of the material in §\ref{Deligne-conjectures-ss}.
	\end{rema}

	\appendix
	
	\section{Neutral Tannakian categories with Frobenius}
	\label{appendix}
	We introduce in this appendix the notion of \textit{neutral Tannakian categories with Frobenius}, and we present a \textit{fundamental exact sequence} for these categories. This formalism applies to the categories of coefficient objects, as explained in Proposition \ref{F-equiv:p}. We have preferred to work here in a more general setting in order to include some other categories, such as the category of \textit{convergent isocrystals}. 
	\subsection{Definition and Weil group}

	\begin{defi}\label{a-defi}
		A \textit{neutral Tannakian category with Frobenius} is a neutral Tannakian category over some field $\KK$, endowed with a $\KK$-linear $\otimes$-autoequivalence $F^*:\widetilde{\bfC}\to \widetilde{\bfC}$. 
	\end{defi}
	
	\begin{cons}\label{a-TCF:cons}
		Let $(\widetilde{\bfC},F^*)$ be a neutral Tannakian category with Frobenius over some field $\KK$. We denote by $\bfC_0$ the category of pairs $(\calE,\Phi)$, where $\calE\in \widetilde{\bfC}$ and $\Phi$ is an isomorphism between $F^*\calE$ and $\calE$. A morphism between two objects $(\calE,\Phi)$ and $(\calE',\Phi')$ is a morphism $f:\calE\to \calE'$ such the following diagram commutes 
		\begin{equation*}
		\begin{tikzcd}
		F^*\calE \arrow[r,"\Phi"]\arrow[d,"F^*f"] &\calE \arrow[d,"f"]\\
		F^*\calE' \arrow[r,"\Phi'"] &\calE'.\
		\end{tikzcd}
		\end{equation*} 
		Let $\Psi:\bfC_0\to \widetilde{\bfC}$ be the functor which sends $(\calE,\Phi)$ to $\calE$. Write $\bfC$ for the smallest Tannakian subcategory of $\widetilde{\bfC}$ containing the essential image of $\Psi$. Choose a fibre functor $\omega$ of $\widetilde{\bfC}$ over $\KK$. It restricts to a fibre functor of $\bfC$ which we will denote by the same symbol. We write $\omega_0$ for the fibre functor of $\bfC_0$ given by the composition $\omega\circ \Psi$. We define $\pi_1(\bfC,\omega)$ and $\pi_1(\bfC_0,\omega_0)$ as the Tannakian groups of $\bfC$ and $\bfC_0$ with respect to $\omega$ and $\omega_0$ respectively. The functor $\Psi$ induces a closed immersion $\pi_1(\bfC,\omega)\hookrightarrow \pi_1(\bfC_0,\omega_0)$ and for every $\calE_0=(\calE,\Phi)\in \bfC_0$ a closed immersion $G(\calE,\omega)\hookrightarrow G(\calE_0,\omega_0)$.
	\end{cons}
	
	\begin{defi}\label{a-cst:d}
		We say that an object in $\bfC_0$ is \textit{constant} if its image in $\bfC$ is trivial, i.e. isomorphic to $\mathbbm{1}^{\oplus n}$ for some $n\in \N$. The constant objects of $\bfC_0$ form a Tannakian subcategory $\bfC_{\cst}\subseteq\bfC_0$. Let $ \pi_1(\bfC_0,\omega_0)^{\cst}$ be the Tannakian group of $\bfC_{\cst}$ with respect to $\omega_0$. The inclusion $\bfC_{\cst}\subseteq \bfC_0$ induces a faithfully flat morphism $\pi_1(\bfC_0,\omega_0)\twoheadrightarrow \pi_1(\bfC_0,\omega_0)^{\cst}$. For every object $\calE_0\in \bfC_0$, we denote by $G(\calE_0,\omega_0)^{\cst}$ the Tannakian group of the full subcategory $\langle\calE_0\rangle_{\cst}\subseteq\langle\calE_0 \rangle$ of constant objects. This induces a faithfully flat morphism $G(\calE_0,\omega_0)\twoheadrightarrow G(\calE_0,\omega_0)^{\cst}$.
	\end{defi}

	\begin{cons}[Weil group]
		\label{a-Weil:c}
	If $(\widetilde{\bfC},F^*)$ admits an isomorphism of fibre functors $\eta:\omega \Rightarrow \omega\circ F^*$, then the group $\pi_1(\bfC,\omega)$ is endowed with an automorphism $\varphi$ which is constructed in the following way. For every $\KK$-algebra $R$, the automorphism $\varphi$ sends $\alpha\in \pi_1(\bfC,\omega)(R)$ to $\eta_R^{-1} \circ \alpha \circ \eta_R$, where $\eta_R$ is the extension of scalars of $\eta$ from $\KK$ to $R$. Let $W(\bfC_0,\omega_0)$ be the semi-direct product $\pi_1(\bfC,\omega)\rtimes \Z$, as group scheme over $\KK$, where $1\in \Z$ acts on $\pi_1(\bfC,\omega)$ as $\varphi$ acts on $\pi_1(\bfC,\omega)$. We say that $W(\bfC_0,\omega_0)$ is the\textit{ Weil group} of $\bfC_0$.
	\end{cons}
	
	\begin{rema}\label{a-existence-iso:r}
		Thanks to \cite{Del11}, we know that if $\KK$ is algebraically closed, an isomorphism $\eta$ as in §\ref{a-Weil:c} always exists. This is not the case, in general, when $\KK$ is not algebraically closed. For coefficient objects, an isomorphism $\eta$ can be constructed without extending $\KK$, as it is explained in Remark \ref{isomorphism-fibre-funct:r}. 
	\end{rema}
	
	\begin{lemm}\label{a-zar-dense:l}
		Let $(\widetilde{\bfC},F^*)$ be a neutral Tannakian category with Frobenius which admits a fibre functor $\omega$ isomorphic to $\omega\circ F^*$. There exists a natural equivalence of categories $\bfC_0\iso \Rep_{\KK}(W(\bfC_0,\omega_0))$ and a natural morphism $\iota:W(\bfC_0,\omega_0)\to \pi_1(\bfC_0,\omega_0)$ such that the following diagram commutes
		\begin{center}
			\begin{tikzcd}[row sep=
				tiny, column sep=huge]
				& \Rep_{\KK}(\pi_1(\bfC_0,\omega_0)) \arrow[dd,"\iota^*"]\\
				\bfC_0 \arrow{ur}{\sim} \arrow{dr}[swap]{\sim} &              \\
				& \Rep_{\KK}(W(\bfC_0,\omega_0)),
			\end{tikzcd}
		\end{center}
		where the equivalence $\bfC_0\iso \Rep_{\KK}(\pi_1(\bfC_0,\omega_0))$ is the one induced by the fibre functor $\omega_0$.
		In addition, the image of $\iota$ is Zariski-dense in $\pi_1(\bfC_0,\omega_0)$. 
	\end{lemm}
	\begin{proof}
		For every $(\calE,\Phi)\in \bfC_0$, we extend the natural representation of $\pi_1(\bfC,\omega)$ on the vector space $\omega(\calE)$ to a representation of $W(\bfC_0,\omega_0)$. Write $e$ for the identity point in $\pi_1(\bfC,\omega)(\KK)$. We impose that $(e,1)\in W(\bfC_0,\omega_0)(\KK)$ acts on $\omega(\calE)$ via $\omega(\Phi)\circ \eta_{\calE}$, where $\eta_{\calE}$ is the isomorphism induced by $\eta$ between $\omega(\calE)$ and $\omega(F^*\calE)$. This defines an equivalence $\bfC_0\iso \Rep_{\KK}(W(\bfC_0,\omega_0))$ and a morphism $\iota: W(\bfC_0,\omega_0)\to \pi_1(\bfC_0,\omega_0)$ satisfying the required properties. By the Tannaka reconstruction theorem, the affine group $\pi_1(\bfC_0,\omega)$ is the pro-algebraic completion of $W(\bfC_0,\omega_0)$, thus the image of $\iota$ is Zariski-dense in $\pi_1(\bfC_0,\omega)$.
	\end{proof}

	\subsection{The fundamental exact sequence}
	\label{a-fundamental-exact-sequence:ss}

	\begin{lemm}\label{a-normal:l}
		Let $(\widetilde{\bfC},F^*)$ be a neutral Tannakian category with Frobenius and let $\omega$ be a fibre functor of $\widetilde{\bfC}$. The subgroup $\pi_1(\bfC,\omega)\subseteq \pi_1(\bfC_0,\omega_0)$ is a normal subgroup. In particular, for every $\calF\in \bfC$ there exists $\calG_0\in \bfC_0$ such that $\calF\subseteq \Psi(\calG_0)$.
	\end{lemm}
	
	\proof
	Thanks to Theorem \cite[Theorem A.1]{EHS}, the second part of the statement follows from the first one. We may verify that the subgroup is normal after extending the field $\KK$ to its algebraic closure. Under the additional assumption that $\KK$ is algebraically closed, by Remark \ref{a-existence-iso:r} there exists an isomorphism between $\omega$ and $\omega\circ F^*$, so that we can construct the Weil group $W(\bfC_0,\omega_0)$ as explained in §\ref{a-Weil:c}. By Lemma \ref{a-zar-dense:l}, the group scheme $W(\bfC_0,\omega_0)$ is endowed with a natural morphism $\iota: W(\bfC_0,\omega_0)\to \pi_1(\bfC_0,\omega_0)$ with Zariski-dense image. Let $H$ be the normaliser of $\pi_1(\bfC,\omega)$ in $\pi_1(\bfC_0,\omega_0)$. The group $\pi_1(\bfC,\omega)$ is normal in $W(\bfC_0,\omega_0)$, hence the $\KK$-point $(e,1)\in W(\bfC_0,\omega_0)(\KK)$ normalizes $\pi_1(\bfC,\omega)$. As a consequence, $\iota(e,1)\in \pi_1(\bfC_0,\omega_0)(\KK)$ is contained in $H(\KK)$. The group $W(\bfC_0,\omega_0)$ is generated by $\pi_1(\bfC,\omega)$ and $(e,1)$, thus the image of $\iota$ is contained in $H$. This implies that $H=\pi_1(\bfC_0,\omega_0)$, which shows that $\pi_1(\bfC,\omega)$ is normal in $\pi_1(\bfC_0,\omega_0)$, as we wanted.
	
	\qed
	
	\begin{theo}\
		\label{a-fundamental-exact-sequence:t}Let $(\widetilde{\bfC},F^*)$ be a neutral Tannakian category over $\KK$ with Frobenius and let $\omega$ be a fibre functor of $\widetilde{\bfC}$. The following statements hold.
		\begin{itemize}
			\item[{(i)}]
			The morphisms constructed in §\ref{a-TCF:cons} and §\ref{a-cst:d} form an exact sequence 
			$$1\to\pi_1(\bfC,\omega)\to\pi_1(\bfC_0,\omega_0) \to \pi_1(\bfC_0,\omega_0)^{\cst} \to 1.$$
			
			\item[{(ii)}]
			For every $\calE_0=(\calE,\Phi)\in\bfC_0$ and every $\calF\in \langle\calE\rangle$, there exists $\calG_0\in \langle \calE_0\rangle$ such that $\calF\subseteq \Psi(\calG_0)$.
			\item[{(iii)}]
			For every object $\calE_0=(\calE,\Phi)\in \bfC_0$, the exact sequence of (i) sits in a commutative diagram with exact rows
			
			\begin{equation*}
			\begin{tikzcd}
			\label{a-fundamental-diagram}
			1 \arrow{r}& \pi_1(\bfC,\omega) \arrow{r}\arrow[d,two heads] &\pi_1(\bfC_0,\omega_0) \arrow{r}\arrow[d,two heads]& \pi_1(\bfC_0,\omega_0)^{\cst} \arrow{r}\arrow[d,two heads]& 1\\
			1\arrow{r} & G(\calE,x) \arrow{r} & G(\calE_0,x)  \arrow{r} & G(\calE_0,x)^{\cst}\arrow{r} & 1,\
			\end{tikzcd}
			\end{equation*}
			where the vertical arrows are the natural quotients.

			\item[{(iv)}] The affine group scheme $\pi_1(\bfC_0,\omega_0)^{\cst}$ is isomorphic to the pro-algebraic completion of $\Z$ over $\KK$ and $G(\calE_0,x)^{\cst}$ is a commutative algebraic group. 
		\end{itemize}
	\end{theo}
	
	\begin{proof}
		We already know that the sequence of part (i) is exact on the left and on the right. It remains to show the exactness in the middle using Theorem \cite[Theorem A.1]{EHS}. Condition (a) is satisfied by construction. For condition (b) we notice that a $\otimes$-functor sends trivial objects to trivial objects. Therefore, for every $(\calE,\Phi)\in \bfC_0$, the maximal trivial subobject $\calF\subseteq \calE$ is sent by $F^*$ to the maximal trivial subobject of $F^*(\calE)$. This means that the restriction of $\Phi$ to $F^*(\calF)$ defines an isomorphism between $F^*(\calF)$ and $\calF$ that we denote by $\Phi|_{\calF}$. The pair $(\calF,\Phi|_{\calF})$ is the subobject of $(\calE,\Phi)$ with the desired property. Condition (c) is proven in Lemma \ref{a-normal:l}.
		
		For part (ii) we notice that the subgroup $G(\calE,\omega) \subseteq G(\calE_0,\omega_0)$ is a quotient of $\pi_1(\bfC,\omega)\subseteq\pi_1(\bfC_0,\omega_0)$, thus it is normal. By Theorem \cite[Theorem A.1]{EHS}, this implies the desired result. The diagram of part (iii) is obtained by taking the natural morphisms of the Tannakian groups. To prove that the lower sequence is exact we proceed as in part (i), replacing Lemma \ref{a-normal:l} with part (ii). Finally, the category $\bfC_{\cst}$ is equivalent to $\Rep_{\KK}(\Z)$, thus $\pi_1(\bfC_0,\omega_0)^{\cst}$ is isomorphic to the pro-algebraic completion of $\Z$ over $\KK$. In particular, for every $\calE_0\in \bfC_0$, the algebraic group $G(\calE_0,\omega_0)^{\cst}$, being a quotient of $\pi_1(\bfC_0,\omega_0)^{\cst}$, is commutative. This concludes the proof.
	\end{proof}

	\bibliographystyle{ams-alpha}

\end{document}